\documentclass{amsart}
\usepackage{float}

\newcommand{\IN}{\mathbb N}

\newcommand{\IR}{\mathbb R}
\newcommand{\w}{\omega}
\newcommand{\lar}{\mathrm{cov}}

\newcommand{\I}{\mathcal I}

\newcommand{\SSS}{\Sigma}

\newcommand{\e}{\varepsilon}

\newtheorem{theorem}{Theorem}[section]
\newtheorem{proposition}[theorem]{Proposition}
\newtheorem{corollary}[theorem]{Corollary}
\newtheorem{problem}[theorem]{Problem}

\newtheorem{lemma}[theorem]{Lemma}
\theoremstyle{definition}
\newtheorem{definition}[theorem]{Definition}

\title{On partitions of $G$-spaces and $G$-lattices}
\author{Taras Banakh, Oleksandr Ravsky, Sergiy Slobodianiuk}
\address{T.Banakh: Ivan Franko University of Lviv (Ukraine) and Jan Kochanowski Unversity in Kielce (Poland)}
\address{O.Ravsky: Pidstyhach Institute for Applied Problems of Mechanics and Mathematics of National Acedemy of Science of Ukraine, Lviv}
\address{S. Slobodianiuk: Taras Schevchenko National University of Kyiv, Ukraine}
\email{t.o.banakh@gmail.com, oravsky@mail.ru, slobodianiuk@yandex.ru}
\begin{document}

\begin{abstract}
Given a $G$-space $X$ and a non-trivial $G$-invariant ideal $\I$ of subsets of $X$, we prove that for every partition $X=A_1\cup\dots\cup A_n$ of $X$ into $n\ge 2$ pieces there is a piece $A_i$ of the partition and a finite set $F\subset G$ of cardinality $|F|\le \phi(n+1):=\max_{1<x<n+1}\frac{x^{n+1-x}-1}{x-1}$
such that $G=F\cdot \Delta(A_i)$ where $\Delta(A_i)=\{g\in G:gA_i\cap A_i\notin\I\}$ is the difference set of the set $A_i$. Also we investigate the growth of the sequence $\phi(n)=\max_{1<x<n}\frac{x^{n-x}-1}{x-1}$ and show that $\ln \phi(n+1)=nW(ne)-2n+\frac{n}{W(ne)}+\frac{W(ne)}{n}+O\big(\frac{\ln\ln n}n\big)$ where $W(x)$ is the Lambert W-function, defined implicitly as $W(x)e^{W(x)}=x$. This shows that $\phi(n)$ grows faster that any exponent $a^n$ but slower than the sequence $n!$ of factorials.
\end{abstract}
\maketitle

%\tableofcontents

\section{Motivation, principal problems and results}

This paper was motivated by the following open problem posed by I.V.~Protasov in the Kourovka Notebook \cite[13.44]{Kourov}.

\begin{problem}\label{prob1} Is it true that for any partition $G=A_1\cup\dots \cup A_n$ of a group $G$ into $n$ pieces there is a piece $A_i$ of the partition such that $G=FA_iA_i^{-1}$ for some finite set $F\subset G$ of cardinality $|F|\le n$?
\end{problem}

A simple measure-theoretic argument shows that the answer to this problem is affirmative for any amenable group $G$. So, the problem actually concerns non-amenable groups. Let us recall that a group $G$ is amenable if it admits a left-invariant finitely additive probability measure $\mu:\mathcal P(X)\to[0,1]$ defined on the Boolean algebra $\mathcal P(X)$ of all subsets of $X$. In Theorem~12.7 of \cite{PB} Protasov and Banakh gave a partial answer to Problem~\ref{prob1} proving that for any partition $G=A_1\cup\dots \cup A_n$ of a group $G$ into $n$ pieces there is a piece $A_i$ of the partition such that $G=FA_iA_i^{-1}$ for some finite set $F\subset G$ of cardinality $|F|\le2^{2^{n-1}-1}$. They also observed that the answer to Problem~\ref{prob1} is affirmative for $n\le 2$.

In \cite{Prot} Protasov considered an ``idealized'' version of Problem~\ref{prob1}. A family $\I$ of subsets of a set $X$ is called an {\em ideal} on $X$ if for any sets $A,B\in\I$ and $C\in\mathcal P(X)$ we get $A\cup B\in\I$ and $A\cap C\in \I$. An ideal $\I$ on $X$ is trivial if $X\in\mathcal I$.

Now assume that $X$ is a $G$-space (i.e., a set endowed with a left action of a group $G$) and $\I$ is a $G$-invariant ideal on $X$. The $G$-invariantness of the ideal $\I$ means that for every $g\in G$ and $A\in \I$ the shift $gA$ of the set $A$ belongs to the ideal $\I$. For a subset $A\subset X$ let $\Delta(A)=\{g\in G:gA\cap A\notin\I\}$ be the $\I$-difference set of $A$.
In \cite{Prot} Protasov asked the following modification of Problem~\ref{prob1}.

\begin{problem}\label{prob2} Let $X$ be an infinite $G$-space and $\I$ be the ideal of finite subsets of $X$. Is it true that for any partition $X=A_1\cup\dots \cup A_n$ of $X$ there is a piece $A_i$ of the partition such that $G=F\cdot\Delta(A_i)$ for some finite set $F\subset G$ of cardinality $|F|\le n$?
\end{problem}

The answer to this problem is affirmative if $X$ admits a $G$-invariant probability measure.
Also the upper bound $2^{2^{n-1}-1}$ on $|F|$ from Theorem 12.7 \cite{PB} generalizes to the ``idealized'' setting, see \cite{Erde}.
Let us observe that Problem~\ref{prob2} actually concerns partitions of the Boolean algebra $\mathcal P(X)/\I$, so it is natural to consider this problem in context of Boolean algebras or more generally, bounded lattices.

By a {\em lattice} we understand a set $X$ endowed with two commutative idempotent associative operations $\vee,\wedge:X\times X\to X$ connected by the absorption law: $x\vee (x\wedge y)=x$ and $x\wedge (x\vee y)=x$ for all $x,y\in X$. Each lattice $(X,\vee,\wedge)$ carries a natural partial order $\le$ in which $x\le y$ iff $x\wedge y=x$ iff $x\vee y=y$. A lattice $X$ is {\em bounded} if it has the smallest element $\mathbf 0$ and the largest element $\mathbf 1$. In the sequel we shall assume that $\mathbf 0\ne\mathbf 1$. This happens if and only if $|X|>1$. A (bounded) lattice is called {\em distributive} (resp. {\em $\mathbf{0}$-distributive}) if for any points $x,y,z\in X$ (with $x\wedge y=\mathbf 0$) we get $x\wedge (y\vee z)=(x\wedge y)\vee (x\vee z)$. For a finite subset $A=\{a_1,\dots,a_n\}$ of a lattice $X$ we put $\bigvee A=a_1\vee\dots\vee a_n$ and $\bigwedge A=a_1\wedge\dots\wedge a_n$.
For an element $a\in X$ of a lattice $X$ and a natural number $n\in\IN$ the set
$$a/n=\{A\subset X:|A|\le n\mbox{ and }\textstyle{\bigvee}A=a\}$$
can be thought as the family of $n$-element covers of $a$.

By a {\em $G$-lattice} we shall understand a lattice $X$ endowed with an action $\alpha:G\times X\to X$, $\alpha:(g,x)\mapsto gx$, of a group $G$ such that for every $g\in G$ the shift $\alpha_g:x\to gx$ of $X$ is an automorphism of the lattice $X$. For a finite subset $F\subset G$ and an element $a\in X$ we put $$Fa=\{fa:f\in F\}\subset X\mbox{ \ and \ }F\cdot a=\bigvee Fa\in X.$$ A basic example of a distributive bounded $G$-lattice is the Boolean algebra $\mathcal P(X)$ of a $G$-space $X$ or its quotient $\mathcal P(X)/\I$ by some non-trivial $G$-invariant ideal $\I$.

For a bounded $G$-lattice $X$ and an element $a\in X$ let
$$\Delta(a)=\{g\in G:ga\wedge a\ne\mathbf 0\}$$be the difference set of $a$. This set is not empty if and only if $a\ne\mathbf 0$.

For a non-empty subset $D$ of a group $G$ let $$\lar(D)=\min\{|F|:F\subset G\mbox{ and }G=F\cdot D\}$$be the {\em covering number} of $D$ in $G$. If $D=\emptyset$, then we put $\lar(D)$ be equal to the smallest infinite cardinal greater than $|G|$, the cardinality of the group $G$.

On the language of lattices, Problem~\ref{prob2} can be generalized as follows.

\begin{problem}\label{prob3} Let $X$ be a bounded $G$-lattice and $A\subset X$ be a finite subset such that $\bigvee A=\mathbf 1$. Is it true that $\min_{a\in A}\lar(\Delta(a))\le|A|$?
\end{problem}

Again the answer to this problem is affirmative for amenable bounded $G$-lattices. A bounded $G$-lattice $X$  is called {\em amenable} if it possesses a $G$-invariant measure $\mu:X\to[0,1]$.

Let $X$ be a bounded $G$-lattice. A function $\mu:X\to[0,1]$ is called
\begin{itemize}
\item {\em $G$-invariant} if $\mu(ga)=\mu(a)$ for any $g\in G$ and $a\in X$;
\item {\em monotone} if $\mu(a)\le\mu(b)$ for any elements $a\le b$ of the lattice $X$;
\item {\em subadditive} if $\mu(a\vee b)\le \mu(a)+\mu(b)$ for any elements $a,b\in X$;
\item {\em additive} if $\mu(a_1\vee\dots\vee a_n)=\mu(a_1)+\dots+\mu(a_n)$ for any elements $a_1,\dots,a_n\in X$ such that $a_i\wedge a_j=\mathbf 0$ for any indices $1\le i<j\le n$;
\item a {\em density} on $X$ if $\mu$ is a monotone function such that $\mu(\mathbf 0)=0$ and $\mu(\mathbf 1)=1$;
\item a {\em submeasure} on $X$ if $\mu$ is a subadditive density on $X$;
\item a {\em measure} on $X$ if $\mu$ is an additive submeasure on $X$.
\end{itemize}

For any density $\mu:X\to[0,1]$ on a bounded lattice $X$ and any natural number $n\in\IN$ the function $$\partial^n\mu:X\to[0,1],\;\;\partial^n\mu:x\mapsto\sup_{A\in x/n}\Big(\mu(x)-\sum_{a\in A}\mu(a)\Big),$$
will be called the {\em $n$-th subadditivity defect of $\mu$}. In this definition
$$x/n=\{A\subset X:|A|\le n\mbox{ \ and \ }\textstyle{\bigvee}A=x\}.$$
For any natural numbers $n\le m$ the inclusion $\{x\}=x/1\subset x/n\subset x/m$ implies that $$0\le \partial^n\mu(x)\le\partial^m\mu(x)\le 1\mbox{ \ for every \ }x\in X.$$
It follows that for any elements $a_1,\dots,a_n\in X$ and their supermum $a=\bigvee_{i=1}^n a_i$ we get
$$\mu(a)\le\partial^n\mu(a)+\sum_{i=1}^n\mu(a_i).$$

The definition of the subadditivity defects implies the following characterization of subadditive densities.

\begin{proposition} A density $\mu:X\to[0,1]$ on a bounded lattice $X$
\begin{enumerate}
\item is subadditive if and only if $\partial^2\mu\equiv 0$ if and only if $\partial^n\mu\equiv 0$ for every $n\ge 2$;
\item has $\partial^n\mu(\mathbf 1)=0$ for all $n\in\IN$ if $\mu\ge\nu$ for some submeasure $\nu:X\to[0,1]$.
\end{enumerate}
\end{proposition}

In turns out that Problems~\ref{prob1}--\ref{prob3} are related to the problem of evaluating the subadditivity defects of the {\em Protasov density} $p_X:X\to[0,1]$ defined on each bounded $G$-lattice $X$ by the formula
$$p_X(a)=\begin{cases}
\dfrac1{\lar(\Delta(a))},&\mbox{if $0<\lar(\Delta(a))<\w$};\\
0,&\mbox{otherwise}.
\end{cases}
$$

The definitions of the Protasov density and the subadditivity defect imply the following simple:

\begin{proposition}\label{p1n} Let $X$ be a bounded $G$-lattice and $n\in\IN$ be a natural number. If $\partial^np_X(\mathbf 1)=0$, then for each subset $A\subset X$ with $|A|\le n$ and $\bigvee A=\mathbf 1$, we get $$\sum_{a\in A}p_X(a)\ge 1\mbox{ \ \ and \ \ }\min_{a\in A}\lar(\Delta(a))=\frac1{\max  p_X|A}\le n.$$
\end{proposition}

This proposition suggests another open problem.

\begin{problem} Let $X$ be a bounded $G$-lattice. Is $\partial^n p_X(\mathbf 1)=0$ for every natural number $n\in\IN$?
\end{problem}

The answer to this problem is affirmative for amenable bounded $G$-lattices and will be given with help of the {\em upper Banach density} $\bar u:X\to[0,1]$ defined on each bounded $G$-lattice $X$ by the formula
$$
\bar u_X(a)=\sup_{\mu}\inf_{g\in G}\mu(ga),
$$
where $\mu$ runs over all measures on $X$. If $X$ has no measure, then we define the Banach density $\bar u:X\to[0,1]$ letting $\bar u_X(\mathbf 1)=1$ and $\bar u_X(a)=0$ for all $a\in X\setminus\{\mathbf 1\}$. It is known \cite{Coquand} that each distributive lattice possesses a measure.

It turns out that the upper Banach density $\bar u_X$ bounds from below the Protasov density $p_X$.

\begin{theorem}\label{t1} For any bounded $G$-lattice $X$ we get $p_X\ge\bar u_X$.
\end{theorem}

\begin{proof} Given any element $a\in X$, we should prove that $\bar u_X(a)\le p_X(a)$. Assuming that $\bar u_X(a)>p_X(a)$, we conclude that $a\notin\{\mathbf 0,\mathbf 1\}$ and $\bar u_X(a)>0$, which implies that the set $M(X)$ of measures on $X$ is not empty and hence $p_X(a)<\bar u_X(a)=\sup_{\mu\in M(X)}\inf_{g\in G}\mu(ga)$. Then we can choose $\e>0$ and a measure $\mu:X\to[0,1]$ such that $\inf_{g\in G}\mu(ga)\ge p_X(a)+\e$. By Zorn's Lemma, there is a maximal subset $F\subset G$ such that $xa\wedge ya=\mathbf 0$ for any distinct elements $x,y\in F$.
The maximality of the set $F$ implies that for every $x\in G$ there is an element $y\in F$ such that  $ya\wedge xa\ne\mathbf 0$, which implies that $a\wedge y^{-1}x\cdot a\ne\mathbf 0$. By the definition of the difference set $\Delta(a)$, we get  $y^{-1}x\in\Delta(a)$ and hence $x\in y\cdot\Delta(a)\subset F\cdot\Delta(a)$. So, $G=F\cdot\Delta(a)$ and $\lar(\Delta(a))\le|F|$.
By the additivity of the measure $\mu$, for any finite subset $E\subset F$ we get $$1=\mu(\mathbf 1)\ge\mu\big(\textstyle{\bigvee\limits_{x\in E}}xa)=\sum_{x\in E}\mu(xa)\ge|E|\cdot \inf_{x\in E}\mu(xa)\ge |E|\cdot (p_X(a)+\e),$$
which implies that $F$ is a finite set of cardinality $|F|\le 1/(p_X(a)+\e)$. Then $$p_X(a)=\frac1{\lar(\Delta(a))}\ge \frac1{|F|}\ge p_X(a)+\e>p_X(a),$$
which is a desired contradiction.
\end{proof}

\begin{corollary} If a bounded $G$-lattice $X$ is amenable, then
 $\partial^np_X(\mathbf 1)=\partial^n\bar u_X(\mathbf 1)=0$ for every $n\in\IN$.
\end{corollary}

\begin{proof} Fix a $G$-invariant measure $\mu:X\to[0,1]$ on $X$ and observe that for every $x\in X$ we get
$$\mu(x)=\inf_{g\in G}\mu(gx)\le\bar u_X(x)\le p_X(x)$$according to Theorem~\ref{t1}. Then for every $n\in\IN$ and a set $A\in \mathbf 1/n$ the subadditivity of the measure $\mu$ implies: $$1=\mu(\mathbf 1)=\mu\big({\textstyle\bigvee_{a\in A}}a\big)\le\sum_{a\in A}\bar u_X(a)\le\sum_{a\in A}p_X(a).$$
Then $0\le \partial^n p_X(\mathbf 1)=\sup_{A\in \mathbf 1/n}(1-\sum_{a\in A}p_X(a))\le 0$ and hence $\partial^np_X(\mathbf 1)=0$. By the same reason $\partial^n\bar u_X(\mathbf 1)=0$.
\end{proof}

\begin{problem} Is a distributive bounded $G$-lattice $X$ amenable if $\partial^n p_X(\mathbf 1)=0$ for all $n\in\IN$?
\end{problem}

By \cite[\S5]{Ban}, for any amenable group $G$ the upper Banach density $\bar u_X:\mathcal P(G)\to[0,1]$ on the Boolean algebra $X=\mathcal P(G)$ is subadditive (and coincides with the right Solecki density considered in \cite{Ban}) and hence has subadditivity defects $\partial^n\bar u_X=0$ for all $n\in\IN$.
However, for non-amenable groups, the Banach density can be highly non-subadditive: by \cite[3.2]{Ban} the free group $G=F_2$ with two generators can be written as the union $G=A\cup B$ of two sets with $\bar u_{X}(A)=\bar u_X(B)=0$. This implies $\partial^n\bar u_X(\mathbf 1)=1$ for all $n\ge 2$, where $\mathbf 1=G$ is the unit of the Boolean algebra $X=\mathcal P(G)$.

The Protasov density $p_X:\mathcal P(G)\to[0,1]$ fails to be subadditive even
for nice (abelian) groups. If $G=A\oplus B$ for infinite subgroups $A,B\subset G$, then the sets $A,B\in\mathcal P(G)=X$ have Protasov density $p_X(A)=p_X(B)=0$ while their union has $p_X(A\cup B)=1$.
This yields $\partial^2p_X(A\cup B)=1$.

Nonetheless the Protasov density has certain weak subadditivity property at $\mathbf 1$. To describe this property in quantitative terms, consider the function
$$\phi:\IN\to\IR,\;\;\phi:n\mapsto \sup_{1<x<n}\frac{x^{n-x}-1}{x-1}.$$
For $n=1$ we put $\phi(1)=0$.

The main result of this paper is the following theorem, which generalizes and improves Theorem 12.7 \cite{PB} and Theorem 1 of \cite{Erde}. This theorem follows from Theorems~\ref{t3} and \ref{bound:s} discussed below.

\begin{theorem}\label{main} For any $\mathbf 0$-distributive bounded $G$-lattice $X$ and any subset $A\subset X$ of finite cardinality $|A|=n\in\IN$ with $\bigvee A=\mathbf 1$ there is an element $a\in A$ with $\lar(\Delta(a))\le \phi(n+1)$ and $p_X(a)\ge\frac1{\phi(n+1)}$.
\end{theorem}

This theorem yields the following upper bound on the subadditivity defects of the Protasov density $p_X$ at the unit $\mathbf 1$ on any $\mathbf 0$-distributive bounded $G$-lattice $X$.

\begin{corollary}\label{t1.8n} For any $\mathbf 0$-distributive bounded $G$-lattice $X$ the Protasov density $p_X:X\to[0,1]$ has the subadditivity defect $$\partial^n p_X(\mathbf 1)\le 1-\dfrac{1}{\phi(n+1)}\mbox{ \ for every \ }n\in\IN.$$
\end{corollary}

In light of these results it is important to evaluate the growth of the function $\phi(n)$ as $n\to\infty$. This will be done in Section~\ref{s:phi} with the help of the Lambert W-function, which is inverse  to the function $y=xe^x$. So, $W(y)e^{W(y)}=y$ for each positive real numbers $y$.
It is known \cite{Lambert} that at infinity the Lambert W-function $W(x)$ has asymptotical growth
$$W(x)=L-l+\frac{l}{L}+\frac{l(-2+l)}{2L^2}+\frac{l(6-9l+2l^2)}{6L^3}+\frac{l(-12+36l-22l^2+3l^3)}{12L^4}+
O\Big[\Big(\frac{l}{L}\Big)^5\Big]$$
where $L=\ln x$ and $l=\ln\ln x$.

The following theorem gives the lower and upper bounds on the (logarithm) of the sequence $\phi(n+1)$  and will be proved in Section~\ref{s:phi}.

\begin{theorem}\label{bound:phi} For every $n\ge 51$
$$nW(ne)-2n+\frac{n}{W(ne)}+\frac{W(ne)}{n}<\ln \phi(n+1)<nW(ne)-2n+\frac{n}{W(ne)}+\frac{W(ne)}{n}+\frac{\ln\ln (ne)}{n}.$$
\end{theorem}

It light of Theorem~\ref{bound:phi}, it is interesting to compare the growth of the sequence $\phi(n)$ with the growth of the sequence $n!$ of factorials. Asymptotical bounds on $n!$ proved  in \cite{factorial} yield the following lower and upper bounds on the logarithm $\ln n!$ of $n!$:
$$n\ln n-n+\frac12\ln n+\frac{\ln 2}2+\frac1{12n+1}<\ln n!<n\ln n-n+\frac{\ln n}n+\frac12\ln n+\frac{\ln 2}{2}+\frac1{12n}.$$
Comparing these two formulas, we see that the sequence $\phi(n)$ grows faster than any exponent $a^n$, $a>1$, but slower than the sequence of factorials.
\smallskip

The the upper bound $\sup_{A\in \mathbf 1/n}\min_{a\in A}\lar(\Delta(a))\le\phi(n+1)$ from Theorem~\ref{main} will be derived from the inequalities $$\sup_{A\in \mathbf 1/n}\min_{a\in A}\lar(\Delta(a))\le s_{-\infty}(n)\le\phi(n+1)$$where the number $s_{-\infty}(n)$ has algorithmic nature and is defined as follows.

Let $\w^n$ be the semigroup of all functions $f:n\to\w$, endowed with the operation of the addition of functions. The semigroup $\w^n$ is partially ordered by the relation $f\le g$ iff $f(i)\le f(i)$ for all $i\in n$. Given two functions $f,g\in\w^n$ we shall write $g<f$ if $g(i)<f(i)$ for all $i\in n$, and put ${\downarrow}f=\{g\in\w^n:g<f\}$ be the {\em strict lower cone} of $f$ in $\w^n$.
In the same way, the set ${\downarrow}\hbar$ can be defined by any function $\hbar:n\to \Omega$ with values in some set $\Omega$ of cardinals. Such functions $\hbar$ will be called {\em cardinal-valued}. For a cardinal-valued function $\hbar:n\to\Omega$  we put ${\downarrow}\hbar=\{g\in\w^n:\forall i\in n\;\;g(i)<\hbar(i)\}$.

For subsets $A_0,\dots,A_{n-1}$ of $\w^n$ let $$\sum_{i\in n}A_i=\Big\{\sum_{i\in n}a_i:\forall i\in n\;\;a_i\in A_i\Big\}$$ be the pointwise sum of the sets $A_0,\dots,A_n$. By $\mathcal P(\w^n)$ we denote the family of all subsets of $\w^n$.

%For a subset $J\subset n$ by $\bar 1_J$ we shall denote the characteristic function of the subset $J$ in $n$. This is the unique function $1_J:n\to\{0,1\}$ such that $\bar 1_J^{-1}(1)=J$. If $J=\{j\}$ is a singleton, then we shall write $1_j$ instead $\bar 1_{\{j\}}$.

Given a cardinal-valued function $\hbar:n\to\Omega$, for every $m\in\w$ consider the functions $\hbar^{\{m\}},\hbar^{[m]}:n\to\mathcal P(\w^n)$ defined by the recursive formulas
$$
\begin{aligned}
&\hbar^{[0]}(i)=\hbar^{\{0\}}(i)=\{1_i\},\\
&\hbar^{\{m+1\}}(i)=\big\{x-x(i)1_i:x\in ({\downarrow} \hbar)\cap \sum_{j\in n}\hbar^{[m]}(j)\big\},\\
&\hbar^{[m+1]}(i)=\hbar^{\{m+1\}}(i)\cup\hbar^{[m]}(i)
\end{aligned}
$$
for $i\in n$ and   $m\in\w$. Let also $\hbar^{[\w]}(i)=\bigcup_{m\in\w}\hbar^{\{m\}}(i)$ for all $i\in n$. The definition of the functions $\hbar^{[k]}$, $k\in\w$, implies that $\hbar^{[\w]}(i)\subset({\downarrow}\hbar)\cup\{1_i\}$ for all $i\in n$, which means that the set $\hbar^{[\w]}(i)$ is finite and is equal to $\hbar^{[k]}(i)$ for some $k\in\w$.

\begin{definition}A cardinal-valued function $\hbar:n\to\Omega$ is called {\em $0$-generating} if the constant zero function $0:n\to\{0\}\subset\w$ belongs to the set $\bigcup_{i\in n}\hbar^{[\w]}(i)$.
 \end{definition}

Let us observe that the problem of recognizing 0-generating functions is algorithmically resolvable.

The following theorem (which will be proved in Section~\ref{s:t2}) is one of two ingredients of the proof of Theorem~\ref{main}.

\begin{theorem}\label{t2} Let $A=\{a_0,\dots,a_{n-1}\}\subset X\setminus\{\mathbf 0\}$ be a finite subset of a $\mathbf{0}$-distributive bounded $G$-lattice $X$ and $\hbar$ be the cardinal-valued function defined by $\hbar(i)=\lar(\Delta(a_i))$ for $i\in n$. If $\sup A=\mathbf 1$, then the function $\hbar$ is not $0$-generating.
\end{theorem}

For a non-zero function $f\in\w^n$ and a real number $q$ let
$$M_q(f)=\Big(\frac1n\sum_{i\in n}f(i)^q\Big)^{\frac1q}$$be the mean value of $f$ of degree $q$. Observe that $M_1(f)$ is the arithmetic mean and $M_{-1}(f)$ is the harmonic mean of the function $f$. For $q=\pm\infty$ we put $$M_{-\infty}(f)=\min_{i\in n}f(i)\mbox{ \ and \ }M_{+\infty}(f)=\max_{i\in n}f(i).$$ It is known that $M_p(f)\le M_q(f)$ for any numbers $-\infty\le p\le q\le+\infty$.

For every $q\in[-\infty,+\infty]$ consider the number
$$s_q(n)=\sup\big\{M_q(\hbar): \hbar\in\w^n\mbox{ is not $0$-generating}\big\}\in[0,+\infty].$$
We shall be especially interested in the numbers $s_{-\infty}(n)$ and $s_{-1}(n)$. These numbers relate as follows:
$$s_{-\infty}(n)\le s_{-1}(n)\le n\cdot s_{-\infty}(n).$$

Theorem~\ref{t2} implies:

\begin{theorem}\label{t3} For every $\mathbf{0}$-distributive bounded $G$-lattice $X$ and every $n\in\IN$ we get
$$\inf_{A\in\mathbf 1/n}\sum_{a\in A}p_X(a)\ge\frac{n}{s_{-1}(n)}\ge\frac1{s_{-\infty}(n)},\quad\quad\partial^np_X(\mathbf 1)\le 1-\frac{n}{s_{-1}(n)}\le 1-\frac1{s_{-\infty}(n)}
$$
and
$$\inf_{A\in\mathbf 1/n}\max_{a\in A}p_X(a)\ge\frac1{s_{-\infty}(n)},\quad\quad\sup_{A\in\mathbf 1/n}\min_{a\in A}\lar(\Delta(a))\le s_{-\infty}(n).
$$
\end{theorem}

The other ingredient of the proof of Theorem~\ref{main} is Theorem~\ref{bound:s} comparing the growth of the sequence $s_{-\infty}(n)$ with growth of the sequences
$$\varphi(n)=\max_{0<k<n}\sum_{i=0}^{n-k-1}k^i=\max_{1<k<n}\frac{k^{n-k}-1}{k-1}\in\IN\mbox{ \ and \ }\phi(n)=\sup_{1<x<n}\frac{x^{n-x}-1}{x-1}\in\IR.$$
It is clear that $\varphi(n)\le\phi(n)$. For $n=1$ we put $\varphi(1)=\phi(1)=0$.

\begin{theorem}\label{bound:s}  For every $n\ge 2$ we have the lower and upper bounds $$\varphi(n)\le\phi(n)<s_{-\infty}(n)\le \varphi(n+1)\le\phi(n+1).$$
\end{theorem}

The upper and lower bound from Theorem~\ref{bound:s} will be proved in Sections~\ref{s:up} and \ref{s:low}, respectively.

Finally, we present the results of computer calculations of the values of the sequences $s_{-\infty}(n)$, $s_{-1}(n)$, $\varphi(n)$ and $1+\lfloor\phi(n)\rfloor$ for  $n\le 9$:

\begin{table}[H]\label{comprez1}
\caption{Values of the numbers $\varphi(n)$, $1+\lfloor\phi(n)\rfloor$, $s_{-\infty}(n)$, $s_{-1}(n)$,  $\varphi(n+1)$, $n!$ for $n\le 9$}
\begin{tabular}{|r|rrrrrrrrr}\hline
$n$ & 1 & 2 & 3 & 4 & 5 & 6 & 7 & 8 & 9\\ \hline
$\varphi(n)$&0&1 &2 &3&7&15&40&121&364  \\
$1+\lfloor\phi(n)\rfloor$&1&2 &3&4&8&17&42&122&395  \\ \hline
$\phantom{{}^{\big|}}s_{-\infty}(n)$&1& 2 & 3 & 5 & 9 & 19&$\le$48&$\le$141&?\\
$\phantom{{}_{\big|}}s_{-1}(n)$&1&2&3&5 &$\ge9\frac{9}{49}$&$\ge19$&?&?&?\\ \hline
$\varphi(n+1)$&1 &2 &3&7&15&40&121&364&1365  \\
$n!$&1 & 2 & 6 &24& 120 &720&4320&30240&241920 \\ \hline
\end{tabular}
\end{table}
\smallskip

Here $\lfloor x\rfloor$ denotes the integer part of the real number $x$.
For $n\le 4$ the values $s_{-\infty}(n)$ and $s_{-1}(n)$ will be calculated in Sections~\ref{s:calc1} and \ref{s:calc2}. \smallskip

Combining the results of computer calculations of the numbers $s_{-\infty}(n)$ for $n\le 5$ with Theorem~\ref{t3}, we get the following values of the subadditivity defects $\partial^n p_X(\mathbf 1)$ of the Protasov density $p_X$ at $\mathbf 1$ on each $\mathbf 0$-distributive bounded $G$-lattice $X$:
\begin{table}[H]
\caption{Values of the numbers $s_{-1}(n)$ and $\partial^np_X(\mathbf 1)$ for $n\le 8$}
\begin{tabular}{|r|rrrrrrrrr}\hline
$n$ & 1 & 2 & 3 & 4 & 5 & 6 & 7 & 8\\ \hline
$\phantom{\Big|}s_{-1}(n)$&1&2&3&5 &$\ge9\frac{9}{49}$&$\ge19$&$\ge42$&$\ge122$\\ \hline
$\phantom{\Big|}\partial^np_X(\mathbf 1)$&0 &0 &0&$\le\frac15$&$ \le \frac{41}{90}$ &$\le\frac{13}{19}$&$\le\frac56$&$\le\frac{57}{61}$  \\ \hline
\end{tabular}
\end{table}

\smallskip

 Theorem~\ref{bound:s} gives the lower and upper bounds on $s_{-\infty}(n)$:
$$\varphi(n)\le1+\lfloor\phi(n)\rfloor\le s_{-\infty}(n)\le \varphi(n+1)$$
for every $n\in\w$.

\begin{problem} Is $s_{-1}(n)\le\varphi(n+1)$ for all (sufficiently large) numbers $n$?
\end{problem}

Looking at Table 1 (containing the results of computer calculations), we can observe that $s_{-\infty}(n)=s_{-1}(n)$ for $n\le 4$ but $s_{-1}(n)>s_{-\infty}(n)$ for $n=5$. The inequality $s_{-1}(5)\ge9\frac{9}{49}$ follows from the empirical fact that the vector $(9,9,9,9,10)$ is not $0$-generating. On the other hand, the vectors $(9,9,9,10,10)$, $(9,9,9,9,11)$, and $(8,9,9,9,12)$, $(8,8,8,8,23)$ are $0$-generating. % $(7,7,7,9,18)$ is not generating.

\begin{problem} Is $s_{-1}(5)=9\frac{9}{49}$?
\end{problem}

\begin{problem} Is $s_{-\infty}(n)>s_{-1}(n)$ for all sufficiently large $n$? (for all $n\ge 5$)?
\end{problem}

Looking at the results of calculations in Table 1, we can see that $s_{-\infty}(n)$ is more near to the lower bound $\phi(n)$ than to the upper bound $\varphi(n+1)$.

\begin{problem} Is $s_{-\infty}(n)=O(\phi(n))$? Is $s_{-\infty}(n)=(1+o(1))\phi(n)$?
\end{problem}

Now we switch to the proofs of the results announced in the introduction.

\section{Proof of Theorem~\ref{t2}}\label{s:t2}

Let $X$ be a $\mathbf{0}$-distributive $G$-lattice and $A=\{a_0,\dots,a_{n-1}\}\subset X\setminus\{\mathbf 0\}$ be a subset such that $\bigvee_{i\in n}a_i=\mathbf 1$. We need to check that the cardinal-valued function $\hbar$ defined by $\hbar (i)=\lar(\Delta(a_i))$ for $i\in n$ is not $0$-generating.

For a number $k\in\IN$ by $[G]^{<k}=\{F\subset G:|F|<k\}$ we shall denote the family of all at most $(k-1)$-element subsets of $G$. For every $i\in n$ and a finite set $F\in [G]^{<\hbar(i)}$ by the definition of  $\lar(\Delta(a_i))=\hbar(i)$ there is a point $v_i(F)\in G\setminus \big(F\cdot\Delta(a_i)\big)$.
It follows that for every $u\in F$ we get $v_i(F)\notin u\cdot\Delta(a_i)$ and hence $u^{-1}v_i(F)\, a_i\wedge a_i=\mathbf 0$ and $a_i\wedge v_i(F)^{-1}u\, a_i=\mathbf 0$. The assignment $v_i:F\mapsto v_i(F)$ determines a function $v_i:[G]^{<\hbar(i)}\to G$ such that
$$
a_i\wedge v_i(F)^{-1}u\, a_i=\mathbf 0\mbox{ for every $u\in F\in [G]^{<\hbar(i)}$}.
$$
Now ${\mathbf 0}$-distributivity of the lattice $X$ guarantees that
\begin{equation}\label{shift}
a_i\wedge v_i(F)^{-1}F\cdot a_i=\mathbf 0\mbox{ for every set $F\in [G]^{<\hbar(i)}$}.
\end{equation}
We recall that $F\cdot a=\bigvee_{f\in F}fa$.

For every $i\in n$ consider the function $\delta_{i}:n\to\mathcal P(G)$ defined by
$$\delta_i(j)=\begin{cases}
\{e_G\}&\mbox{if $i=j$},\\
\emptyset&\mbox{if $i\ne j$},
\end{cases}
$$where $e_G$ denotes the neutral element of the group $G$.
Let us recall that $\hbar^{\{0\}}(i)=\{1_i\}$ and define the function $\Phi^{\{0\}}_i:\hbar^{\{0\}}(i)\to \mathcal P(G)^n$ letting $\Phi^{\{0\}}_i(1_i)=\delta_i\in \mathcal P(G)^n$.
Observe that for the unique point $x=1_i$ of the set $\hbar^{\{0\}}(i)$ and the function $\Psi=\Phi_i^{\{0\}}(x)=\delta_i$ the following two conditions hold:
\begin{enumerate}
\item[(1$_0$)] $|\Psi(j)|\le x(j)$ for all $j\in n$;
\item[(2$_0$)] $a_i\le \bigvee_{j\in n}\Psi(j)\cdot a_j$.
\end{enumerate}

By induction for every $i\in\w$ and $m\ge 1$ we shall construct a function
$$\Phi^{\{m\}}_i:\hbar^{\{m\}}(i)\to\mathcal P(G)^n$$such that
for every $x\in \hbar^{\{m\}}(i)$ and the function $\Psi=\Phi_i^{\{m\}}(x)\in \mathcal P(G)^n$ the following conditions hold:
\begin{enumerate}
\item[(1$_m$)] $|\Psi(k)|\le x(k)$ for all $k\in n$;
\item[(2$_m$)] $a_i\le \bigvee_{k\in n}\Psi(k)\cdot a_k$.
\end{enumerate}

Assume that for some $m\ge 1$ and all $i\in n$ and $k<m$ the functions  $\Phi^{\{k\}}_i:\hbar^{\{k\}}(i)\to\mathcal P(G)^n$ have been constructed. Now for every $i\in n$ we shall define the function $\Phi^{\{m\}}_i$. Given any vector $x\in \hbar^{\{m\}}(i)$, find a function $y\in ({\downarrow}\hbar)\cap\sum_{j\in n}\hbar^{[m-1]}(j)$ such that $x=y-y(i)1_i$.
It follows that $y=\sum_{j\in n}y_j$ for some functions $y_j\in \hbar^{[m-1]}(j)$, $j\in n$.
For every $j\in n$ find a number $m_j<m$ such that $y_j\in \hbar^{\{m_j\}}(j)$.
By the inductive hypothesis, for every $j\in n$ the function $\Psi_j=\Phi^{\{m_j\}}_j(y_j)\in\mathcal P(G)^n$ has two properties:
\begin{enumerate}
\item[(1$_{m_j}$)] $|\Psi_j(k)|\le y_j(k)$ for all $k\in n$;
\item[(2$_{m_j}$)] $a_j\le \bigvee_{k\in n}\Psi_j(k)\cdot a_k$.
\end{enumerate}

Now consider the function $$\Upsilon=\bigcup_{j\in n}\Psi_j:n\to\mathcal P(G),\;\;\Upsilon: k\mapsto\bigcup_{j\in n}\Psi_j(k).$$
It follows that for every $k\in n$ the set $\Upsilon(k)\in\mathcal P(G)$ has cardinality
$$|\Upsilon(k)|\le\sum_{j\in n}|\Psi_j(k)|\le \sum_{j\in n}y_j(k)=y(k)<\hbar(k).$$
In particular, $|\Upsilon(i)|<\hbar(i)$. So, $\Upsilon(i)\in[G]^{<\hbar(i)}$ and the element $g_i=v_i(\Upsilon(i))\in G$ is well-defined and by (\ref{shift}) has the property
\begin{equation}\label{cond0}
a_i\wedge g_i^{-1}\Upsilon(i)\cdot a_i=\mathbf 0.
\end{equation}

Finally consider the function $\Psi:n\to\mathcal P(G)^n$ defined by
$$\Psi(k)=\begin{cases}
g_i^{-1}\Upsilon(k)&\mbox{if $k\ne i$}\\
\emptyset&\mbox{if $k=i$}
\end{cases}
$$and put $\Phi^{\{m\}}_i(x)=\Psi$. It follows that so defined function $\Psi$ has the property (1$_m$) of the inductive construction because for every $k\in n$ with $k\ne i$ we get
$$|\Psi(k)|=|g_i^{-1}\Upsilon(k)|=|\Upsilon(k)|\le y(k)=x(k)$$and
$0=|\emptyset|=|\Psi(i)|\le x(i)$.

Next, we check that $\Psi$ also satisfies the condition (2$_m$) of the inductive construction.
The conditions (2$_{m_j}$) applied to functions $\Psi_j$, $j\in n$, guarantee that
$$\mathbf 1=\bigvee_{j\in n}a_j\le \bigvee_{j\in n}\bigvee_{k\in n}\Psi_j(k)\cdot a_k=\bigvee_{k\in n}\bigvee_{j\in n}\Psi_j(k)\cdot a_k=\bigvee_{k\in n}\bigcup_{j\in n}\Psi_j(k)\cdot a_k=\bigvee_{k\in n}\Upsilon(k)\cdot a_k$$and hence $$\mathbf 1= \bigvee_{k\in n}g_i^{-1}\Upsilon(k)\cdot a_k.$$ The $\mathbf{0}$-distributivity of the lattice $X$ and the condition (\ref{cond0}) imply that
$$
\begin{aligned}
a_i\wedge \mathbf 1&=a_i\wedge\Big(\bigvee_{k\in n}g_i^{-1}\Upsilon(k)\cdot a_k\Big)=\Big(a_i\wedge g_i^{-1}\Upsilon(i)\cdot a_i)\vee \Big(a_i\wedge \bigvee_{i\ne k\in n}g_i^{-1}\Upsilon(k)\cdot a_k\Big)=\\
&=\mathbf 0\vee \Big(a_i\wedge \bigvee_{i\ne k\in n}\Psi(k)\cdot a_k\Big)\le a_i\wedge \Big(\bigvee_{k\in n}\Psi(k)\cdot a_k\Big),
\end{aligned}
$$which implies that $a_i\le\bigvee_{k\in n}\Psi(k)\cdot a_k$ and completes the inductive construction.

Now we can complete the proof of Theorem~\ref{t2}. Assuming that the function $\hbar$ is $0$-generating, we would conclude that the zero function $z:n\to\{0\}$ belong to the set $\hbar^{\{m\}}(i)$ for some $m\in\w$ and $i\in n$. For the function $z$, consider the function $\Psi=\Phi^{\{m\}}_i(z)$. For this function, the conditions (1$_m$), (2$_m$), $m\in\w$, of the inductive construction yield:
\begin{enumerate}
\item[(1$_z$)] $|\Psi(k)|\le z(k)=0$ for all $k\in n$;
\item[(2$_z$)] $a_i\le \bigvee_{k\in n}\Psi(k)\cdot a_k=\bigvee\emptyset=\mathbf 0$,
\end{enumerate}
which contradicts the choice of the element $a_i\in X\setminus\{\mathbf 0\}$.

\section{Characterizing constant $0$-generating functions}

In this section we prove Theorem~\ref{t0n} characterizing constant $0$-generating functions. This theorem will be used in Section~\ref{s:up} for the proof of the upper bound $s_{-\infty}(n)\le\varphi(n+1)$ from Theorem~\ref{bound:s}.

Fix an integer number $n\ge 2$. We consider the set $\w^n$ as a $G$-space endowed with the natural right action $\w^n\times G\to\w^n$, $(f,\sigma)\mapsto f\circ \sigma$, of the group $G=\SSS_n$ of all permutations of the set $n=\{0,\dots,n-1\}$. For a function $f\in\w^n$ by
$$\|f\|=\max_{i\in n}f(i)$$we denote its norm.

For a subset $J\subset n$ by $\bar 1_J:n\to\{0,1\}$ we denote the characteristic function of the set $J$. This is a unique function such that $\bar 1_J^{-1}(1)=J$.

 For a subset $A\subset \w^n$ and a number $k\in\w$ by $\sum^kA$ we denote the set-sum of $k$ copies of $A$. If $k=0$, then $\sum^0 A=\{\mathbf 0\}$ is the singleton consisting of the constant zero function $\mathbf 0\in\w^n$. Let also $A\circ\Sigma_n=\{f\circ\sigma:f\in A,\;\sigma\in\Sigma_n\}$ and ${\uparrow}A=\{f\in\w^n:\exists g\in A$ with $g\le f\}$. On the other hand, ${\downarrow}f=\{g\in\w^n:g<f\}$ for a function $f\in\w^n$. We shall identify integer numbers $c\in\IN$ with the constant functions $\hbar_c:n\to\{c\}\subset \w$.

Given a constant function $\hbar\in \w^n$ consider the sequence of finite subsets $\hbar^{(m]}\subset\w^n$, $m\in\w$, defined inductively as $\hbar^{(0]}=\emptyset$ and
$$\hbar^{(m+1]}=\hbar^{(m]}\cup \big\{(x-x(n{-}1)\cdot 1_{n{-}1})\circ\sigma:\sigma\in\Sigma_n,\;x\in ({\downarrow}\hbar)\cap\bigcup_{0\le k<n}\bar 1_{n\setminus k}+\textstyle{\sum^{k}}\hbar^{(m-1]}\big\}\mbox{ \ for \ }m\in\w.
$$

\begin{theorem}\label{t0n} A constant function $\hbar\in\w^n$ is $0$-generating if and only if the constant zero function $\mathbf 0:n\to\{0\}$ belongs to the set $\hbar^{(\w]}=\bigcup_{m\in\w}\hbar^{(m]}$.
\end{theorem}

\begin{proof} Let $\hbar:n\to\w$ be a constant function. To prove the theorem it suffices to check that
$$\bigcup_{i\in n}\hbar^{\{m\}}(i)\subset{\uparrow}\hbar^{(m]}\subset\bigcup_{i\in n}{\uparrow}\hbar^{[m]}(i)$$for every $m\in\IN$. This will be done in Lemmas~\ref{l3n} and \ref{l4n}, which will be proved with the help of Lemmas~\ref{l1n} and \ref{l2n}.

\begin{lemma}\label{l1n} For every permutation $\sigma\in S_n$ and $m\in\w$ we get $$\hbar^{\{m\}}(i)\circ\sigma\subset\hbar^{\{m\}}(\sigma^{-1}(i))\mbox{ \ for all \ }i\in n.$$
\end{lemma}

\begin{proof} This lemma will be proved by induction on $m$. For $m=0$ and every $i\in n$ the set $\hbar^{\{0\}}(i)$ contains a unique element $1_i$, for which $1_i\circ\sigma=1_{\sigma^{-1}(i)}$. So, $\hbar^{\{0\}}(i)\circ\sigma=\{1_{\sigma^{-1}(i)}\}=\hbar^{\{0\}}(\sigma^{-1}(i))$.

Assume that the lemma has been proved for all numbers smaller or equal than some $m\in\w$. To show that $\hbar^{\{m+1\}}(i)\circ\sigma\subset\hbar^{\{m+1\}}(\sigma^{-1}(i))$ for all $i\in n$, take any function $f\in\hbar^{\{m+1\}}(i)$ and find functions $g_j\in\hbar^{[m]}(j)$, $j\in n$, such that the function $g=\sum_{j\in n}g_j$ is strictly smaller than $\hbar$ and $f=g-g(i)1_i$. By the inductive assumption, for every $j\in n$ the function $g_j\circ\sigma$ belongs to the set $\hbar^{[m]}(\sigma^{-1}(j))$. This implies that for every $k\in n$ the function $h_k=g_{\sigma(k)}\circ \sigma$ belongs to $\hbar^{[m]}(k)$. It follows that the function $h=\sum_{k\in n}h_k=\sum_{k\in n}g_{\sigma(k)}\circ\sigma=g\circ\sigma<\hbar\circ\sigma=\hbar$. Consequently, for every $i\in n$ the function $h-h(\sigma^{-1}(i))1_{\sigma^{-1}(i)}$ belongs to $\hbar^{\{m+1\}}(\sigma^{-1}(i))$. Now observe that
$$h\circ\sigma^{-1}=\Big(\sum_{k\in n}h_k\Big)\circ\sigma^{-1}=
\Big(\sum_{k\in n}g_{\sigma(k)}\circ\sigma\Big)\circ\sigma^{-1}(i)=\sum_{k\in n}g_{\sigma(k)}=g$$ and
$h\circ\sigma^{-1}(i)=g(i)$. So,
$$f\circ\sigma=(g-g(i)1_i)\circ\sigma=g\circ\sigma-g(i)1_{\sigma^{-1}(i)}=h-h(\sigma^{-1}(i))1_{\sigma^{-1}(i)}\in\hbar^{[m]}(\sigma^{-1}(i))$$
and we are done.
\end{proof}

\begin{lemma}\label{l2n} For every $m\in\IN$, permutation $\sigma\in \Sigma_n$, index $i\in n$ and a non-zero function $f\in \hbar^{\{m\}}(i)$ the function $f\circ\sigma$ belongs to the set ${\uparrow}\hbar^{[m]}(j)$ for every index $j\in n$.
\end{lemma}

\begin{proof} If $f\circ\sigma(j)>0$, then $f\circ\sigma\ge 1_j$ and hence $f\circ\sigma\in {\uparrow}\hbar^{[0]}(j)$. So, we assume that $f\circ\sigma(j)=0$. If $\sigma^{-1}(i)=j$, then $f\circ \sigma\in \hbar^{\{m\}}(\sigma^{-1}(i))\subset \hbar^{[m]}(j)$ by Lemma~\ref{l1n}. So, we assume that $\sigma^{-1}(i)\ne j$. It follows from $f\in\hbar^{\{m\}}(i)$ that $f(i)=0$. Let $\tau\in \Sigma_n$ be the permutation such that $\tau^{-1}(j)=\tau(j)=\sigma^{-1}(i)$ and $\tau(k)=k$ for any $k\in n\setminus\{j,\sigma^{-1}(i)\}$. Lemma~\ref{l1n} implies that $f\circ\sigma\circ\tau\in \hbar^{\{m\}}((\sigma\circ\tau)^{-1}(i))=
\hbar^{\{m\}}(j)$. It remains to check that $f\circ\sigma=f\circ\sigma\circ\tau$.

Fix any index $k\in n$. If $k\notin\{j,\sigma^{-1}(i)\}$, then $f\circ\sigma\circ\tau(k)=f\circ\sigma(k)$. If $k=j$, then $f\circ\sigma\circ\tau(j)=f\circ\sigma(\sigma^{-1}(i))=f(i)=0=f\circ\sigma(j)$.
If $k=\sigma^{-1}(i)$, then $f\circ\sigma\circ\tau(k)=f\circ\sigma(j)=0=f(i)=f\circ\sigma(k)$.
\end{proof}

\begin{lemma}\label{l3n} $\bigcup_{i\in n}\hbar^{\{m\}}(i)\subset{\uparrow}\hbar^{(m]}$ for every $m\ge 1$.
\end{lemma}

\begin{proof} First we check the lemma for $m=1$. In this case for every $i\in n$ the set $\hbar^{\{1\}}(i)$ consists of a single function $x$, which coincides with the characteristic function $\bar 1_{n\setminus\{i\}}$ of the set $n\setminus\{i\}$. Let $\sigma\in\Sigma_n$ be the transposition exchanging $i$ and $n-1$. Then $$x=\bar 1_{n-1}\circ\sigma=(\bar 1_n-\bar 1_n(n-1)\cdot 1_{n-1})\circ\sigma\in \hbar^{(1]}.$$

Now assume that the lemma has been proved for all numbers smaller or equal than some $m\in\IN$.
To prove the lemma for $m+1$, take any $i\in n$ and a function $x\in \hbar^{\{m+1\}}(i)$. By the definition of the set $\hbar^{\{m+1\}}(i)$ there is a function $y\in({\downarrow}\hbar)\cap\sum_{j\in n}\hbar^{[m]}(j)$ such that $x=y-y(i)\cdot 1_i$.
Find functions $y_j\in\hbar^{[m]}(j)$, $j\in n$, such that $y=\sum_{j\in n}y_j$ and consider the set $J=\{j\in n:y_j=1_j\}$. Then $y=\bar 1_J+\sum_{j\in n\setminus J}y_j$.  For every $j\in n\setminus J$ the function $y_j\ne 1_j$ belongs to $\hbar^{\{m_j\}}(j)$ for some positive $m_j\le m$. By the inductive assumption, $y_j\in \hbar^{\{m_j\}}(j)\subset \hbar^{(m_j]}\subset \hbar^{(m]}$.

Choose a permutation $\sigma\in\Sigma_n$ such that $\sigma^{-1}(i)=n-1$ and $\sigma^{-1}(\{i\}\cup J)=n\setminus k$ for some $k\le n$. Separately we shall consider two cases.
\smallskip

1) If $i\in J$, then $n-1=\sigma^{-1}(i)\in\sigma^{-1}(J)=n\setminus k$ and
$$y\circ\sigma=\bar 1_J\circ\sigma+\sum_{j\in n\setminus J}y_j\circ\sigma\in \bar 1_{n\setminus k}+\sum_{j\in n\setminus J}\hbar^{\{m_j\}}(j)\circ\sigma\subset \bar 1_{n\setminus k}+\sum_{j\in n\setminus J}\hbar^{(m]}\circ \sigma=\bar 1_{n\setminus k}+\textstyle{\sum^{k}}\hbar^{(m]}.$$Since $y\circ\sigma\le\|y\circ\sigma\|=\|y\|<\hbar$, we conclude that the function $x\circ\sigma=(y-y(i)\cdot 1_i)\circ\sigma=y\circ\sigma-y\circ\sigma(n-1)1_{n-1}\in \hbar^{(m+1]}$ and hence $x\in \hbar^{(m+1]}\circ \Sigma_n=\hbar^{(m+1]}$.
\smallskip

2) Next, we assume that $i\notin J$. If $y_i\circ\sigma(n-1)=0$, then $y\ge y_i$ implies $$x\circ \sigma=y\circ\sigma-y\circ\sigma(n-1)\cdot 1_{n-1}=y\circ\sigma\ge y_i\circ\sigma\in\hbar^{(m]}\circ\sigma$$and hence $x\in{\uparrow}\hbar^{(m]}$.

If $y_i\circ\sigma(n-1)>0$, then $y_i\circ\sigma\ge 1_{n-1}$ and
$$
\begin{aligned}
y\circ\sigma&=\bar 1_J\circ\sigma+\sum_{j\in n\setminus J}y_j\circ\sigma=\bar 1_{\sigma^{-1}(J)}+y_i\circ\sigma+\sum_{i\ne j\in n\setminus J}y_j\circ\sigma\ge \\
&\ge
\bar 1_{(n-1)\setminus k}+1_{n-1}+\sum_{i\ne j\in n\setminus J}y_j\circ\sigma\ge \bar 1_{n\setminus k}+\sum_{i\ne j\in n\setminus J}\hbar^{\{m_j\}}(j)\circ\sigma\subset\\
 &\subset\bar 1_{n\setminus k}+\sum_{i\ne j\in n\setminus J}\hbar^{(m]}\circ\sigma=
\bar 1_{n\setminus k}+\textstyle{\sum^{k}}\hbar^{(m]}.
\end{aligned}
$$
Since $y\circ\sigma\le\|y\circ\sigma\|=\|y\|<\hbar$, we conclude that $x\circ\sigma=y\circ\sigma-y\circ\sigma(n-1)\cdot 1_{n-1}\in {\uparrow}\hbar^{(m]}$ and then $x\in{\uparrow}\hbar^{(m]}\circ\sigma={\uparrow}\hbar^{(m]}$.
\end{proof}

\begin{lemma}\label{l4n} For every $m\in\w$ and every $i\in n$ we get $\hbar^{(m]}\subset{\uparrow}\hbar^{[m]}(i)$.
\end{lemma}

\begin{proof} For $m=0$ this inclusion is trivial. Assume that the inclusion from the lemma has been proved for some $m\ge 0$. To prove it for $m+1$, take any function $x\in\hbar^{(m+1]}$. If $x\in \hbar^{(m]}$, then $x\in {\uparrow}\hbar^{(m]}(i)\subset {\uparrow}\hbar^{[m]}(i)$ by the inductive assumption. If $x\in\hbar^{(m+1]}\setminus \hbar^{(m]}$, then there is a number $k<n$ and a function $y\in \bar 1_{n\setminus k}+\sum^{k}\hbar^{(m]}$ such that $y<\hbar$ and $x=(y-y(n-1)\cdot 1_{n-1})\circ\sigma$ for some permutation $\sigma\in\Sigma_n$. Write $y$ as the sum $y=\bar 1_{n\setminus k}+\sum_{j\in k}y_j$ for some functions $y_j\in\hbar^{(m]}$, $j\in k$. By the inductive assumption, for every $j\in k$ the function $y_j\in\hbar^{(m]}$ belongs to the set ${\uparrow}\hbar^{[m]}(j)$. Letting $y_j=1_j$ for $j\in k$, we see that $y=\sum_{j\in n}y_j\in \sum_{j\in n}{\uparrow}\hbar^{[m]}(j)$ and hence $y-y(n-1)\cdot 1_{n-1}\in{\uparrow}\hbar^{\{m+1\}}(n-1)$. By Lemma~\ref{l2n}, the function $x=(y-y(n-1)\cdot 1_{n-1})\circ\sigma$ belongs to ${\uparrow}\hbar^{[m+1]}(i)$.
\end{proof}
\end{proof}

\section{The proof of the upper bound $s_{-\infty}(n)\le\varphi(n+1)$ from Theorem~\ref{bound:s}}\label{s:up}

To prove the upper bound $s_{-\infty}(n)\le \varphi(n+1)$ from Theorem~\ref{bound:s}, it suffices to check that for  $n\in\IN$ the constant function $\hbar:n\to\{1+\varphi(n+1)\}$ is $0$-generating.
In order to do that, we shall construct a special double sequence of functions $f_{k,m}\in{\downarrow}\hbar$ defined as follows.

We recall that $$\varphi(n+1)=\max_{0<k\le n}\sum_{i=0}^{n-k} k^i=\max_{0<k<n+1}\frac{k^{n+1-k}-1}{k-1}.$$

For $n=1$ the $0$-generacy of the constant function $\hbar\equiv 1+\varphi(2)=2$ is trivial, so we shall assume that $n\ge 2$. Denote by $\sigma\in \Sigma_n$ the cyclic permutation of $n$ defined by
$$\sigma(i)=\begin{cases}n-1&\mbox{if $i=0$}\\
i-1&\mbox{otherwise}
\end{cases}
$$
and consider the map $\vec S:\w^n\to\w^n$ assigning to each function $f\in\w^n$ the function $\vec Sf=\big(f-f(n-1)\cdot 1_{n-1}\big)\circ\sigma$. It is easy to check that for every $i\in n$ we get
$$
\vec Sf(i)=\begin{cases}
0&\mbox{for $i=0$},\\
f(i-1)&\mbox{for $i>0$}.
\end{cases}
$$

This observation and the definition of the set $\hbar^{(\w]}=\bigcup_{m\in\w}\hbar^{(m]}$ imply:

\begin{lemma}\label{l1nn} For any non-negative $k<n$ and a function $f\in\w^n$ with $\vec S f\in\hbar^{(\w]}$ and $\bar 1_{n\setminus k}+k\cdot\vec Sf<\hbar$ we get $$\vec S(\bar 1_{n\setminus k}+k\cdot\vec Sf)\in \hbar^{(\w]}.$$
\end{lemma}

Let $f_0=\bar 1_{n}$ and for every $0<k\le n$ consider the function $f_k\in\w^n$ defined by
$$
f_k(i)=\begin{cases}0,&\mbox{ if $0\le i<k$},\\
\sum_{j=0}^{i-k}k^{j},&\mbox{ if $k\le i<n$}.
\end{cases}
$$
It follows that $f_n\equiv 0$ and $$f_k(i)=\frac{k^{i-k+1}-1}{k-1}\le \varphi(i+1)\le\varphi(n)<\hbar$$ for $2\le k\le i<n$. We shall put $\frac{k^m-1}{k-1}=m$  for $k=1$ and $m\in\w$.

\begin{lemma} $f_k=\bar 1_{n\setminus k}+k\cdot \vec Sf_k$ for any $0<k\le n$.
\end{lemma}

\begin{proof} If $i<k$, then $f_k(i)=0=\bar 1_{n\setminus k}(i)+k\cdot \vec Sf_k(i)$.

If $i=k$, then $\bar 1_{n\setminus k}(k)+k\cdot \vec Sf_k(k)=1+k\cdot f_k(k-1)=1+k\cdot 0=1=k^0=f_k(k)$.

If $k<i<n$, then $$1_{n\setminus k}(i)+k\cdot \vec Sf_k(i)=1+k\cdot f_k(i-1)=1+k\cdot\sum_{j=0}^{i-1-k}k^{j}=\sum_{j=0}^{i-k}k^{j}=f_k(i).$$
\end{proof}

For every $0<k\le n$ let $f_{k,0}=f_{k-1}$ and $f_{k,m+1}=\bar 1_{n\setminus k}+k\cdot \vec S(f_{k,m})$ for $m\in\w$.

\begin{lemma}\label{l2} For every $0<k\le n$ and $0<m\le n-k+1$ we get
$$
f_{k,m}(i)=\begin{cases}
0&\mbox{if $i<k$}\\
f_k(i)&\mbox{$k\le i<k+m-1$}\\
k^m\cdot \sum_{j=0}^{i-k-m+1}(k-1)^{j}+\sum_{j=0}^{m-1}k^j&\mbox{if $k+m-1\le i<n$}.
\end{cases}
$$
\end{lemma}

\begin{proof} For $m=1$, we get $f_{k,1}=\bar 1_{n\setminus k}+k\cdot \vec Sf_{k-1}$, which implies $f_{k,1}(i)=0$ for $i<k$ and
$$
\begin{aligned}
f_{k,1}(i)&=1+k\cdot f_{k-1}(i-1)=k\cdot\sum_{j=0}^{i-k}(k-1)^{j}+1=k^m\cdot\sum_{j=0}^{i-k-m+1}(k-1)^{j}+\sum_{j=0}^{m-1}k^j
\end{aligned}
$$
for $k=k+m-1\le i<n$.

Assume that the claim has been proved for some $0<m<n-k-1$. To prove it for $m+1$, take any number $i\in n$ and consider the value $f_{k,m+1}(i)=\bar 1_{n\setminus k}(i)+k\cdot \vec Sf_{k,m}(i)$.

If $i=0$, then $f_{k,m+1}(i)=f_{k,m+1}(0)=\bar 1_{n\setminus k}(0)+k\cdot \vec Sf_{k,m}(0)=0+k\cdot 0=0$.

If $0<i<k$, then $f_{k,m+1}(i)=0$ as $\bar 1_{n\setminus k}(i)=0$ and $\vec Sf_{k,m}(i)=f_{k,m}(i-1)=0$ by the inductive assumption.

If $i=k$, then $f_{k,m+1}(k)=\bar 1_{n\setminus k}(k)+k\cdot \vec Sf_{k,m}(k-1)=1+0=\sum_{j=0}^{i-k}k^j=f_k(i)$.

If $k<i<k+(m+1)-1$, then $k\le i-1<k+m-1$ and by the inductive assumption
$$
f_{k,m+1}(i)=\bar 1_{n\setminus k}(i)+k\cdot \vec Sf_{k,m}(i)=1+k\cdot f_{k,m}(i-1)=1+k\cdot \sum_{j=0}^{i-1-k}k^j=\sum_{j=0}^{i-k}k^j=f_k(i).
$$

If $k+(m+1)-1\le i<n$, then $k+m-1\le i-1<n-1$ and then
$$f_{k,m+1}(i)=1+k\cdot f_{k,m}(i-1)=k\cdot\Big(k^m\cdot \sum_{j=0}^{i-k-m}(k-1)^{j}+\sum_{j=0}^{m-1}k^j\Big)+1=k^{m+1}\cdot \sum_{j=0}^{i{-}(m{+}1){-}k{+}1}(k-1)^{j}+\sum_{j=0}^mk^j.
$$
\end{proof}

The following lemma combined with Theorem~\ref{t0n} and the fact that $\vec S f_n=f_n=\mathbf 0$ implies that the constant function $\hbar\equiv \varphi(n+1)+1$ is $0$-generating and hence $s_{-\infty}(n)\le \varphi(n+1)$.

\begin{lemma} For every $0\le k\le n$ the function $\vec Sf_{k}$ belongs to the set $\hbar^{(\w]}$.
\end{lemma}

\begin{proof} The proof is by induction on $k$. For $k=0$ the function $\vec Sf_0=\bar 1_{n\setminus 1}$ belongs to $\hbar^{(1]}\subset\hbar^{(\w]}$ by the definition of $\hbar^{(1]}$. Assume that for some positive number $k<n$ we have proved that the function $\vec S f_{k-1}$ belongs to $\hbar^{(\w]}$.

By induction on $m\le n-k+1$ we shall prove that the function $\vec Sf_{k,m}$ belongs to $\hbar^{(\w]}$. For $m=0$ this follows from the inductive assumption as $f_{k,0}=f_{k-1}$. Assume that for some $m\le n-k+1$ we have proved that $\vec Sf_{k,m}\in\hbar^{(\w]}$. By Lemma~\ref{l2},
$$\|f_{k,m+1}\|=f_{k,m+1}(n-1)=k^{m}\cdot \sum_{j=0}^{n-k-m}(k-1)^{j}+\sum_{j=0}^{m-1}k^j\le
k^{m}\sum_{j=0}^{n-k-m}k^{j}+\sum_{j=0}^{m-1}k^j=\sum_{j=0}^{n-k}k^j\le \varphi(n-m+1)<\hbar.$$
By Lemma~\ref{l1nn}, $\vec S f_{k,m+1}=\bar 1_{n\setminus k}+k\cdot \vec Sf_{k,m}\in\hbar^{(\w]}$.
Thus $\vec S f_{k,m}\in\hbar^{(\w]}$ for all $m\le n-k+1$. In particular, $\vec S f_{k+1}=\vec Sf_{k,n-k+1}\in\hbar^{(\w]}$.
\end{proof}

\section{The proof of the lower bound $\phi(n)<s_{-\infty}(n)$ from Theorem~\ref{bound:s}}\label{s:low}

In this section for every $n\ge 2$ we prove the lower bound $\phi(n)<s_{-\infty}(n)$ from Theorem~\ref{bound:s}.

If $n\le 3$, then $1+\lfloor\phi(n)\rfloor=n$. So, it suffices to check that
$n\le s_{-\infty}(n)$. For this consider any group $G$ of order $n$. The Boolean algebra $\mathcal P(G)$ consisting of all subsets of $G$ is a distributive $G$-lattice. Taking into account that $p_X(A)\ge\frac1{|G|}=\frac1n$ for any non-empty subset $A\subset G$ and $p_X(\{a\})=\frac1{n}$ for any singleton $\{a\}\subset G$, we see that
$$\frac1n=\inf_{A\in\mathbf 1/n}\max_{a\in A}p_X(a)\le \frac1{s_{-\infty}(n)}$$according to Theorem~\ref{t3}, which implies the desired lower bound $s_{-\infty}(n)\ge n>\phi(n)$ for $n\le 3$.

Next, we consider the case $n\ge 4$.
We recall that $\phi(n)$ is the maximum of the function
$$\phi_n(x)=\frac{x^{n-x}-1}{x-1}$$on the interval ${]}1,n]$.
By standard methods of Calculus, it can be shown that the function $\phi_n(x)$ attains its maximal value at a unique point $\lambda\in {]}1,n]$.

Given any positive number $c\le \frac{\lambda^{n-1}-1}{\lambda-1}$, consider the function $\xi_{c}:[1,n]\to\IR$ defined by
$$\xi_{c}(x)=(x-\lambda)c+\frac{\lambda^{n-x}-1}{\lambda-1}$$
and find its minimum. For this observe that
$$\xi_{c}'(x)=c-\frac{\lambda^{n-x}\ln (\lambda)}{\lambda-1}$$is an increasing function, equal to zero at a point $x=x_c$ such that
$$\lambda^{-x}=\frac{c(\lambda-1)}{\lambda^n
\ln(\lambda)}.$$
This implies that at the point $$x_c= n+\frac{\ln\ln(\lambda)-\ln(\lambda-1)-\ln(c)}{\ln(\lambda)}$$ the function $\xi_{c}$ attains its minimal value:
$$
\begin{aligned}
\xi_{c}(x_c)&=(x_c-\lambda)c+\frac{\lambda^{n-{x_c}}-1}{\lambda-1}=
\Big(n-\lambda+\frac{\ln\ln(\lambda)-\ln(\lambda-1)-\ln(c)}{\ln(\lambda)}\Big)c+\frac{c}{\ln(\lambda)}-\frac{1}{\lambda-1}=\\
&=\Big(n-\lambda+\frac{\ln\ln(\lambda)-\ln(\lambda-1)+1}{\ln(\lambda)}\Big)c-\frac{\ln(c)}{\ln(\lambda)}c-\frac{1}{\lambda-1}.
\end{aligned}
$$

Now consider the function
$$\zeta(c)=\min_{1<x<n}\xi_{c}(x)=\xi_{c}(x_c)$$
and find its maximum. This function has derivative:
$$
\zeta'(c)=n-\lambda+\frac{\ln\ln(\lambda)-\ln(\lambda-1)+1}{\ln(\lambda)}-
\frac{\ln(c)}{\ln(\lambda)}-\frac{1}{\ln(\lambda)}
$$
which is a decreasing function, equal to zero at a unique point $c_\lambda$ such that
$$\ln(c_\lambda)=(n-\lambda)\ln(\lambda)+\ln\ln(\lambda)-\ln(\lambda-1) \mbox{ \ and \ }c_\lambda=\frac{\lambda^{n-\lambda}\ln(\lambda)}{\lambda-1}.$$
Consequently, at this point the function $\zeta(c)$ attains its maximal value:
$$
\begin{aligned}
\zeta(c_\lambda)&=\Big(n-\lambda+\frac{\ln\ln(\lambda)-\ln(\lambda-1)+1-\ln(c_\lambda)}{\ln(\lambda)}\Big)c_\lambda-\frac{1}{\lambda-1}=\\
&=
\Big(n-\lambda+\frac{\ln\ln(\lambda)-\ln(\lambda-1)+1-((n-\lambda)\ln(\lambda)+\ln\ln(\lambda)-\ln(\lambda-1))}{\ln(\lambda)}\Big) \frac{\lambda^{n-\lambda}\ln(\lambda)}{\lambda-1} - \frac{1}{\lambda-1}=\\
&=
\frac{1}{\ln(\lambda)} \frac{\lambda^{n-\lambda}\ln(\lambda)}{\lambda-1} -\frac{1}{\lambda-1}=\frac{\lambda^{n-\lambda}-1}{\lambda-1}=\phi_n(\lambda).
\end{aligned}
$$
Then for the number
$$c_\lambda=\frac{\lambda^{n-\lambda}\ln(\lambda)}{\lambda-1}$$ we get
$$(k-\lambda)c_\lambda+\frac{\lambda^{n-k}-1}{\lambda-1}\ge\min_{1<x<n}\xi_{c_\lambda}(x)=
\zeta(c_\lambda)=\phi_n(\lambda)=\phi(n)$$
for every $1<k<n$. This inequality can be rewritten in the form
\begin{equation}
\label{crucial}
\frac{1}{\lambda}\Big(-\phi(n)+\frac{\lambda^{n-k}-1}{\lambda-1}+kc_\lambda\Big)\ge c_\lambda
\end{equation}
which will be used in the proof of the lower bound $\phi(n)\le s(n)$ from Theorem~\ref{bound:s}.

\begin{lemma}\label{l5.1} If $n\ge 4$, then
$$c_\lambda\le\frac{\lambda^{n-1}-1}{\lambda-1}.
$$
\end{lemma}

\begin{proof} For $n\le 7$ the inequality from lemma can be verified by computer calculations, which give the following results:
\smallskip

$$
\begin{array}{|r|rrrrrrrrr} \hline
\phantom{\Big|}n=                            &   3 &   4&   5&    6&    7&  8\\ \hline
\phantom{\Big|}\lambda\approx               &0.49 &1.48&1.93& 2.34& 2.72&3.07\\
\phi_n(\lambda)\approx       &1.29 &3.51&7.01&16.01&41.53&121.31\\
c_\lambda\approx             &0.23 &2.19&5.32&14.24&42.14&136.61  \\
\phantom{\Big|}\frac{\lambda^{n-1}-1}{\lambda-1}\approx&-0.17&2.48&5.48&19.26&86.61&456.78\\
\hline
\end{array}
$$
\smallskip

If $n\ge 8$, then the function $\phi_n(x)$ is increasing at $x=3$, which implies that $\lambda^{n-1}\ge\lambda>3$ and then
$$
\frac{\lambda^{n-1}-1}{c_\lambda(\lambda-1)}=\frac{\lambda^{n-1}-1}{\lambda^{n-\lambda}\ln(\lambda)}\ge
\frac{\lambda^{n-1}-\frac12\lambda^{n-1}}{\lambda^{n-\lambda}\ln(\lambda)}=\frac{\lambda^{\lambda-1}}{2\ln(\lambda)}\ge\frac{\lambda^2}{2\ln(\lambda)}>1.
$$
\end{proof}

With the help of the real numbers $\lambda$ and $c_\lambda$, we can introduce the notion of {\em weight} $w(f)$ of a function $f\in\w^n$ letting
$$w(f)=\min_{\sigma\in \SSS_n}\sum_{i=0}^{n-1}\lambda^i\cdot f\circ \sigma(i).$$
Here $\Sigma_n$ denote the group of all permutations of the set $n=\{0,\dots,n-1\}$.
The definition of the weight $w$ implies:

\begin{lemma} The weight $w:\w^n\to\IR$ is a monotone and $\SSS_n$-invariant function on $\w^n$.
\end{lemma}

The lower bound $\phi(n)<s_{-\infty}(n)$ will be proved as soon as we check that the constant function
$$\hbar:n\to\{1+\lfloor\phi(n)\rfloor\}\subset\w$$ is not $0$-generating. This is done in the following lemma.

\begin{lemma} For any $m\in\IN$ and any $x\in \bigcup_{i\in n}\hbar^{\{m\}}(i)$ we get $w(x)\ge c_\lambda>0$, which implies that $x\ne 0$ and $\hbar$ is not $0$-generating.
\end{lemma}

\begin{proof} The proof is by induction on $m\in\w$. For $m=1$ and every $i\in n$ the set $\hbar^{\{1\}}(i)$ consists of a unique function $x$, which coincides with the characteristic function $\bar 1_{n\setminus\{i\}}$ of the set $n\setminus\{i\}$ and has weight
$$w(x)=\sum_{j=0}^{n-2}\lambda^j=\frac{\lambda^{n-1}-1}{\lambda-1}\ge c_\lambda$$according to Lemma~\ref{l5.1}.

Assume that the lemma was proved for some $m\ge 0$. To prove it for $m+1$, take any function $x\in \bigcup_{i\in n}\hbar^{\{m+1\}}(i)$. We need to check that $w(x)\ge c_\lambda$.
Find an index $i\in n$ such that $x\in h^{\{m+1\}}(i)$.

By the definition of $h^{\{m+1\}}(i)$, there are functions $y_i\in \hbar^{[m]}(j)$, $j\in n$, such that the sum $y=y_0+\dots+y_{n-1}$ is strictly smaller than $\hbar$ and $x=y-y(i)\cdot 1_i$.
Taking into account that $y$ is an integer-valued function with $y<1+\lfloor \phi(n)\rfloor$, we conclude that $y\le\phi(n)$.
Replacing $y$ by $y\circ \sigma$ for a suitable permutation $\sigma\in \SSS_n$ we can assume that
$w(y)=\sum_{i\in n}\lambda^i\cdot y(i)$. In this case the function $y$ is non-increasing.
Let $K=\{j\in n:y_j=1_j\}$ and put $k=|K|$. Observe that the characteristic function $\bar 1_K:n\to\{0,1\}$  of the set $K\subset n$ has weight
$$w(\bar 1_K)=w(\bar 1_k)=\sum_{i=0}^{k-1}\lambda^i=\frac{\lambda^k-1}{\lambda-1}.$$

Since $y$ is non-increasing, $y(0)$ is the maximal value of the function $y\le\phi(n)$ and then
$$
\begin{aligned}
w(x)&=w\big(y-y(i)\cdot 1_i\big)\ge w\big(y-y(0)\cdot 1_0\big)= \sum_{i=1}^{n-1}\lambda^{i-1}y(i)=\frac1{\lambda}\Big(-y(0)+\sum_{i=0}^{n-1}\lambda^iy(i)\Big)>\\
&>\frac1{\lambda}\Big(-\phi(n)+\sum_{i=0}^{n-1}\lambda^i\sum_{j=0}^{n-1}y_j(i)\Big)=
\frac{1}{\lambda}\Big(-\phi(n)+\sum_{j\in K}\sum_{i=0}^{n-1}\lambda^iy_j(i)+
\sum_{j\in n\setminus K}\sum_{i=0}^{n-1}\lambda^iy_j(i)\Big)\ge\\
&\ge\frac{1}{\lambda}\Big(-\phi(n)+\sum_{i=0}^{n-1}\lambda^i\sum_{j\in K}1_j(i)+
\sum_{j\in n\setminus K}w(y_j)\Big)\ge\frac{1}{\lambda}\Big(-\phi(n)+\sum_{i=0}^{n-1}\lambda^i\bar 1_K(i)+
\sum_{j=n\setminus K}c_\lambda\Big)\ge\\
&\ge\frac{1}{\lambda}\Big(-\phi(n)+w(\bar 1_K)+
(n-k)c_\lambda\Big)\ge
\frac{1}{\lambda}\Big(-\phi(n)+\frac{\lambda^k-1}{\lambda-1}+
(n-k)c_\lambda\Big)\ge
c_\lambda
\end{aligned}
$$
according to the inequality (\ref{crucial}).
\end{proof}

\section{Proof of Theorem~\ref{bound:phi}}\label{s:phi}

In this section we shall prove Theorem~\ref{bound:phi} evaluating the growth of the sequence $\phi(n)$.

This will be done with the help of the Lambert W-function $W(x)$, which is the solution of the equation $$W(x)e^{W(x)}=x.$$
This equation is equivalent to
\begin{equation}\label{lam2}
e^{W(x)}=\frac{x}{W(x)}.
\end{equation}
It is easy to check that
\begin{equation}\label{lam3}
\ln x-\ln\ln x<W(x)<\ln x\mbox{ \  for all \ }x>e.
\end{equation}

With the help of the Lambert W-function we shall calculate the maximal value of the function $\psi_n(x)={x^{n-x}}$ which has the same growth order as the function $\phi_{n+1}(x)=\frac{x^{n+1-x}-1}{x-1}$, whose maximum on the interval ${]}1,n+1]$ is equal to $\phi(n+1)$.

\begin{lemma}\label{l6.1} The function $\ln \psi_n(x)=(n-x)\ln x$ attains its maximum
$$nW(ne)-2n+\frac{n}{W(ne)}\mbox{ \ at the point \ }x_\psi=\frac{n}{W(ne)}.$$
\end{lemma}

\begin{proof}
Observe that
$$\frac{d}{dx}\ln\psi_n(x)=\frac{n-x}{x}-\ln x.$$
Consequently the point of maximum of the function $\psi_n(x)$ can be found from the equation
$$0=n-x-x\ln x=n-x\ln(xe).$$
Multiplying this equation by $e$ and substituting $\ln(xe)=y$, we get
$$0=en-xe\ln(xe)=ne-ye^y,$$
which implies that $y=W(ne)$ and
$$xe=e^y=e^{W(ne)}=\frac{ne}{W(ne)}$$according to the equation (\ref{lam2}).

The value of the function $\ln\psi_n(x)=(n-x)\ln(x)$ at the point $x_\psi=\frac{n}{W(ne)}=e^{W(ne)-1}$ equals
$$\Big(n-\frac{n}{W(ne)}\Big)\cdot \big(W(ne)-1\big)=nW(ne)-2n+\frac{n}{W(ne)}.$$
\end{proof}

\begin{lemma}\label{l6.2} If $n\ge 51$, then the function $\phi_{n+1}(x)=\frac{x^{n+1-x}-1}{x-1}$ attains its maximum at a point $x_\phi$ such that
$$\frac{n}{\ln n}+1<x_\phi<\frac{n}{W(ne)}.$$
\end{lemma}

\begin{proof}
It can be shown that the derivative of the function $\phi_{n+1}(x)$:
$$
\begin{aligned}
\phi_{n+1}'(x)&=\frac1{(x-1)^2}
\Big(e^{(n+1-x)\ln(x)}\Big(\frac{n+1-x}{x}-\ln(x)\Big)(x-1)-e^{(n+1-x)\ln(x)}+1\Big)=\\
&=\frac1{(x-1)^2}\Big(e^{(n+1-x)\ln(x)}\big(n+1-x-\frac{n+1}{x}-(x-1)\ln(x)\big)+1\Big)
\end{aligned}
$$
has a unique zero $x_\phi$ (at which the function $\phi_{n+1}(x)$ attains its maximum).

By computer calculations one can show that for $x=\frac{n}{\ln n}+1$ we get
$$
\begin{aligned}
n+1-x-\frac{n+1}{x}-(x-1)\ln(x)&=n-\frac{n}{\ln n}-\frac{(n+1)\ln n}{n+\ln n}-\frac{n}{\ln n}\ln\Big(1+\frac{n}{\ln(n)}\Big)=\\
&=\frac{n}{\ln n}\Big(\ln n-1-\Big(1+\frac1n\Big)\frac{\ln^2n}{n+\ln n}-\ln\Big(\frac{n}{\ln n}+1\Big)\Big)>0
\end{aligned}
$$
if $n\ge 51$. This means that the function $\phi_{n+1}(x)$ is increasing at the point $x=\frac{n}{\ln n}+1$, which implies that $x<x_\phi$.

On the other hand, for the point $x=\frac{n}{W(ne)}=e^{W(ne)-1}$ we get
$$
n+1-x-\frac{n+1}{x}-(x-1)\ln(x)=n+1-\frac{n}{W(ne)}-\frac{n+1}{n}W(ne)-\big(\frac{n}{W(ne)}-1\big)(W(ne)-1)=\\
-\frac{W(ne)}{n}<0,
$$
which implies that $\phi_{n+1}'(x)=\frac1{(x-1)^2}(-x^{n+1-x}\frac1x+1)<0$,
 the function $\phi_{n+1}(x)$ is decreasing at $x=\frac{n}{W(ne)}$ and hence $x_\phi<\frac{n}{W(ne)}$.
\end{proof}

Our strategy is to evaluate the maximum of the function $\phi_{n+1}(x)=(x^{n+1-x}-1)/(x-1)$ using known information on the maximal value of the function $\psi_n(x)=x^{n-x}$. For this we establish some lower and upper bounds on the logarithm of the fraction $\frac{\phi_{n+1}(x)}{\psi_n(x)}$. We recall that $x_\phi$ (resp. $x_\psi$) stands for the point at which the function $\phi_{n+1}(x)$ (resp. $\psi_n(x)$) attains its maximal value. By Lemmas~\ref{l6.1} and \ref{l6.2}, $$x_\psi=\frac{n}{W(ne)}\mbox{ \ \ and \ \ }\frac{n}{\ln n}+1<x_\phi<\frac{n}{W(ne)}.$$

\begin{lemma}\label{l6.3} If $n\ge 51$, then
\begin{enumerate}
\item $\ln\dfrac{\phi_{n+1}(x_\phi)}{\psi_{n}(x_\phi)}<\dfrac{\ln n}{n}.$
\item $\ln\dfrac{\phi_{n+1}(x_\psi)}{\psi_{n}(x_\psi)}>\dfrac{W(ne)}{n}.$
\end{enumerate}
\end{lemma}

\begin{proof} It follows that for $x=x_\phi$ we get
$$
\ln\frac{\phi_{n+1}(x)}{\psi_{n}(x)}=\ln\frac{x^{n+1-x}-1}{x^{n-x}(x-1)}<
\ln\frac{x^{n+1-x}}{x^{n-x}(x-1)}=\ln\Big(1-\frac1{x-1}\Big)<\frac1{x-1}<\frac{\ln n}{n}
$$according to Lemma~\ref{l6.2}.

On the other hand, the inequality $n\ge 51>2e$ implies that for the point $x=x_\psi=n/W(ne)=e^{W(ne)-1}$ of maximum of the function $\psi_{n}(x)$ we get $W(ne)e^{W(ne)}=ne\ge 2e^2$. In this case $W(ne)\ge 2$ and
$$n+1-x=n+1-\frac{n}{W(ne)}\ge n+1-\frac{n}2>3$$and hence $x^{n+1-x}>x^{3}$. Also $x=e^{W(ne)-1}\ge e$ implies that
$$\frac12-\frac1{x}-\frac1{2x^2}\ge\frac12-\frac{1}{e}-\frac1{2e^2}>0.$$
Using the known lower bound $\ln(1+z)>z-\frac12z^2$ holding for all $z>0$, we conclude that
$$
\begin{aligned}
\ln\frac{\phi_{n+1}(x)}{\psi_{n}(x)}&=\ln\frac{x^{n+1-x}-1}{x^{n-x}(x-1)}=
\ln\Big(\frac{1-x^{x-n-1}}{1-x^{-1}}\Big)>\ln\Big(\frac{1-x^{-3}}{1-x^{-1}}\Big)=
\ln\Big(1+\frac1x+\frac1{x^2}\Big)>\\
&>\frac1x+\frac1{x^2}-\frac12\Big(\frac1x+\frac1{x^2}\Big)^2=
\frac1x+\frac1{x^2}\Big(\frac12-\frac1{x}-\frac{1}{2x^2}\Big)\ge\frac1x+\frac1{x^2}\Big(\frac12-\frac1{e}-\frac{1}{2e^2}\Big)>
\frac1x=\frac{W(ne)}{n}.
\end{aligned}
$$
\end{proof}

Now Theorem~\ref{bound:phi} follows from:

\begin{lemma} For every $n\ge 51$ we get
\begin{enumerate}
\item $\ln\phi(n+1)>nW(ne)-2n+\frac{n}{W(ne)}+\frac{W(ne)}{n}$;
\item $\ln \phi(n+1)<nW(ne)-2n+\frac{n}{W(ne)}+\frac{W(ne)}{n}+\frac{\ln\ln (ne)}{n}$.
\end{enumerate}
\end{lemma}

\begin{proof} 1. By Lemmas~\ref{l6.1} and \ref{l6.3}(2),
$$\ln\phi(n+1)=\ln\phi_{n+1}(x_\phi)\ge\ln\phi_{n+1}(x_\psi)=\ln\psi_{n}(x_\psi)+
\ln\frac{\phi_{n+1}(x_\psi)}{\psi_{n}(x_\psi)}>nW(ne)-2n+\frac{n}{W(ne)}+\frac{W(ne)}{n}.
$$

2. By Lemmas~\ref{l6.1} and \ref{l6.3}(1),
$$
\begin{aligned}
\ln\phi(n+1)&=\ln\phi_{n+1}(x_\phi)=\ln\psi(x_\phi)+\ln\frac{\phi_{n+1}(x_\phi)}{\psi(x_\phi)}<
\ln\psi(x_\psi)+\frac{\ln n}{n}=\\
&=nW(ne)-2n+\frac{n}{W(ne)}+\frac{W(ne)}{n}-\frac{W(ne)}{n}+\frac{\ln n}{n}.
\end{aligned}$$

It remains to find an upper bound on the difference
$\frac{\ln n}{n}-\frac{W(ne)}{n}.$ Taking into account that $W(ne)>\ln(ne)-\ln\ln(ne)$ we see
that
$$
\frac{\ln n}{n}-\frac{W(ne)}{n}<
\frac{\ln n}{n}-\frac{1+\ln(n)-\ln\ln(ne))}{n}<\frac{\ln\ln(ne)}{n}.
$$
\end{proof}

\section{Evaluating the numbers $s_{-\infty}(n)$ for $n\le 5$}\label{s:calc1}

In this section we shall calculate the values of the numbers $s_{-\infty}(n)$, $n\le 5$, from Table~\ref{comprez1}. Each function $x\in\w^n$ will be identified with the sequence $(x(0),\dots,x(n-1))$.

\subsection{Lower bounds} Theorem~\ref{bound:s} yields the lower bound $1+\lfloor\phi(n)\rfloor\le s_{-\infty}(n)$ which is equal to $s_{-\infty}(n)$ for $n\le 3$. For $n=4$ this does not work as $1+\lfloor\phi(n)\rfloor=4$ while $s_{-\infty}(4)=5$. To see that $s_{-\infty}(4)\ge 5$, consider the set $$M_4=\{(0,0,1,2),(0,0,0,4)\}\circ\Sigma_4\subset\w^4.$$ By routine calculations it can be shown that for the constant function $\hbar:4\to\{5\}\subset \w$ we get
$$\big\{(x-x(3)1_3)\circ\sigma:\sigma\in \Sigma_4,\;x\in ({\downarrow}\hbar)\cap\bigcup_{0\le k<4}\big(\bar 1_{4\setminus k}+\textstyle{\sum^k} M_4\big)\big\}\subset{\uparrow}M_4.$$
This implies $\hbar^{(\w]}\subset{\uparrow}M_4$ and $(0,0,0,0)\notin\hbar^{(\w]}$. Then Theorem~\ref{t0n} guarantees that the constant function $\hbar:4\to\{5\}\subset \w$ is not $0$-generating and hence $s_{-\infty}(4)\ge 5$.
\smallskip

For $n=5$ the inequality $s_{-\infty}(n)\ge 9$ follows from the observation that for the set
$$M_5=\{(0,0,1,1,2),(0,0,0,1,6),(0,0,0,2,4),(0,0,0,3,3)\}\circ\Sigma_5$$and the constant function $\hbar:5\to\{9\}\subset\w$ we get
$$\big\{(x-x(4)\cdot 1_4)\circ\sigma:\sigma\in \Sigma_5,\;x\in ({\downarrow}\hbar)\cap\bigcup_{0\le k<5}\big(\bar 1_{5\setminus k}+\textstyle{\sum^k} M_5\big)\big\}\subset{\uparrow}M_5.$$

\subsection{Upper bounds}
According to Theorem~\ref{t0n}, to show that $s_{-\infty}(n)<\hbar$ for some constant $\hbar\in\IN$, it suffices to find a sequence of functions $(f_i)_{i=1}^m$ such that $f_m$ is the zero function and each function $f_i$, $1\le i\le m$, is equal to $(\hat f_i-\hat f_i(n-1)\cdot 1_{n-1})\circ\sigma$ for some permutation $\sigma\in\Sigma_n$ and some function $\hat f_i\in \bigcup_{0\le k< n}\big(\bar 1_{n\setminus k}+\sum^k\{f_j\}_{1\le j<i}\big)$ with $\hat f_i<\hbar$.
\medskip

1) For $n=1$ the inequality $s_{-\infty}(1)\le 1$ is witnessed by the sequence $(f_i)_{i=1}^1$ of length 1:
\smallskip

\begin{table}[H]
\caption{A witness for $s_{-\infty}(1)\le 1$}
\begin{tabular}{|c|c|c|c|}
\hline
$f_i$&$\hat f_i$&$\bar 1_{n\setminus k}+\sum_{j\in k}f_j\phantom{\Big|}$&$k$\\ \hline
(0)&(1)&(1)&0\\
\hline
\end{tabular}
\end{table}

\medskip

2) For $n=2$ the inequality   $s_{-\infty}(2)\le 2$ is witnessed by the sequence $(f_i)_{i=1}^2$ of length 2:
\smallskip

\begin{table}[H]
\caption{A witness for $s_{-\infty}(2)\le 2$}
\begin{tabular}{|c|c|l|c|}
\hline
$f_i$&$\hat f_i$&$\bar 1_{n\setminus k}+\sum_{j\in k}f_j\phantom{\Big|}$&$k$\\ \hline
(1,0)&(1,1)&(1,1)&0\\
(0,0)&(0,2)&(0,1)+(0,1)&1\\
\hline
\end{tabular}
\end{table}

3) For $n=3$ the sequence witnessing that $s_{-\infty}(3)\le 3$ has length 4:
\smallskip

\begin{table}[H]
\caption{A witness for $s_{-\infty}(3)\le 3$}
\begin{tabular}{|c|c|l|c|}
\hline
$f_i$&$\hat f_i$&$\bar 1_{n\setminus k}+\sum_{j\in k}f_j\phantom{\Big|}$&$k$\\ \hline
 (1,1,0)&(1,1,1)&(1,1,1)&0\\
 (0,2,0)&(0,2,2)&(0,1,1)+(0,1,1)&1\\
 (0,1,0)&(1,1,3)&(0,1,1)+(0,0,2)&1\\
 (0,0,0)&(0,0,3)&(0,0,1)+(0,0,1)+(0,0,1)&2\\
\hline
\end{tabular}
\end{table}
\medskip

4) For $n=4$ the sequence witnessing that $s_{-\infty}(4)\le 5$ has length 8:
\smallskip

\begin{table}[H]
\caption{A witness for $s_{-\infty}(4)\le 5$}
\begin{tabular}{|c|c|l|c|}
\hline
$f_i$&$\hat f_i$&$\bar 1_{n\setminus k}+\sum_{j\in k}f_j\phantom{\Big|}$&$k$\\ \hline
(1,1,1,0)&(1,1,1,1)&(1,1,1,1)&0\\
(0,2,2,0)&(0,2,2,2)&(0,1,1,1)+(0,1,1,1)&1\\
(0,1,3,0)&(0,1,3,3)&(0,1,1,1)+(0,0,2,2)&1\\
(0,1,2,0)&(0,1,2,4)&(0,1,1,1)+(0,0,1,3)&1\\
(0,0,3,0)&(0,0,3,5)&(0,0,1,1)+(0,0,1,2)+(0,0,1,2)&2\\
(0,1,1,0)&(0,1,1,4)&(0,1,1,1)+(0,0,0,3)&1\\
(0,0,2,0)&(0,0,2,5)&(0,0,1,1)+(0,0,1,1)+(0,0,0,3)&2\\
(0,0,0,0)&(0,0,0,5)&(0,0,1,1)+(0,0,0,2)+(0,0,0,2)&2\\
\hline
\end{tabular}
\end{table}

\medskip

5) For $n=5$ the sequence witnessing that $s_{-\infty}(5)\le 9$ has length 23 and is presented in Table~7.

\begin{table}\label{tab:up5}
\caption{A witness for $s_{-\infty}(5)\le 9$}
\begin{tabular}{|c|c|l|l|}
\hline
$f_i$&$\hat f_i$&$\bar 1_{n\setminus k}+\sum_{j\in k}f_j\phantom{\Big|}$&$k$\\ \hline
(1,1,1,1,0)&(1,1,1,1,1)&(1,1,1,1,1)&0\\
(0,2,2,2,0)&(0,2,2,2,2)&(0,1,1,1,1)+(0,1,1,1,1)&1\\
(0,1,3,3,0)&(0,1,3,3,3)&(0,1,1,1,1)+(0,0,2,2,2)&1\\
(0,1,2,4,0)&(0,1,2,4,4)&(0,1,1,1,1)+(0,0,1,3,3)&1\\
(0,1,2,3,0)&(0,1,2,3,5)&(0,1,1,1,1)+(0,0,1,2,4)&1\\
(0,0,3,5,0)&(0,0,3,5,7)&(0,0,1,1,1)+(0,0,1,2,3)+(0,0,1,2,3)&2\\
(0,1,1,4,0)&(0,1,1,4,6)&(0,1,1,1,1)+(0,0,0,3,5)&1\\
(0,0,3,3,0)&(0,0,3,3,9)&(0,0,1,1,1)+(0,0,1,1,4)+(0,0,1,1,4)&2\\
(0,0,1,7,0)&(0,0,1,7,7)&(0,0,1,1,1)+(0,0,0,3,3)+(0,0,0,3,3)&2\\
(0,1,1,2,0)&(0,1,1,2,8)&(0,1,1,1,1)+(0,0,0,1,7)&1\\
(0,0,2,4,0)&(0,0,2,4,9)&(0,0,1,1,1)+(0,0,1,2,1)+(0,0,0,1,7)&2\\
(0,0,1,5,0)&(0,0,1,5,9)&(0,0,1,1,1)+(0,0,0,2,4)+(0,0,0,2,4)+&2\\
(0,1,1,2,0)&(0,1,1,2,8)&(0,1,1,1,1)+(0,0,0,1,5)&1\\
(0,0,2,3,0)&(0,0,2,3,8)&(0,0,1,1,1)+(0,0,1,1,2)+(0,0,0,1,5)&2\\
(0,0,1,4,0)&(0,0,1,4,9)&(0,0,1,1,1)+(0,0,0,1,5)+(0,0,0,2,3)&2\\
(0,0,1,3,0)&(0,0,1,3,9)&(0,0,1,1,1)+(0,0,0,1,4)+(0,0,0,1,4)&2\\
(0,0,2,2,0)&(0,0,2,2,9)&(0,0,1,1,1)+(0,0,0,1,3)+(0,0,1,0,3)&2\\
(0,0,0,5,0)&(0,0,0,5,9)&(0,0,0,1,1)+(0,0,0,1,3)+(0,0,0,1,3)+(0,0,0,2,2)&3\\
(0,0,1,2,0)&(0,0,1,2,9)&(0,0,1,1,1)+(0,0,0,1,3)+(0,0,0,0,5)&2\\
(0,0,0,4,0)&(0,0,0,4,9)&(0,0,0,1,1)+(0,0,0,1,2)+(0,0,0,2,1)+(0,0,0,0,5)&3\\
(0,0,0,3,0)&(0,0,0,3,9)&(0,0,0,1,1)+(0,0,0,1,2)+(0,0,0,1,2)+(0,0,0,0,4)&3\\
(0,0,0,2,0)&(0,0,0,2,9)&(0,0,0,1,1)+(0,0,0,1,2)+(0,0,0,0,3)+(0,0,0,0,3)&3\\
(0,0,0,0,0)&(0,0,0,0,9)&(0,0,0,0,1)+(0,0,0,0,2)+(0,0,0,0,2)+(0,0,0,0,2)+(0,0,0,0,2)&4\\
\hline
\end{tabular}
\end{table}
\medskip

For $n=6$ the length of the annulating sequence found by computer is equal to 143. So, it is too long to be presented here.

\section{Evaluating the numbers $s_{-1}(n)$ for $n\le 4$}\label{s:calc2}

In this section we calculate the values of the numbers $s_{-1}(n)$ for $n\le 4$, presented in Table~\ref{comprez1}. We recall that
$$s_{-1}(n)=\sup\big\{M_{-1}(x):x\in\w^n\mbox{ is not $0$-generating}\big\}$$
is the maximal value of the harmonic means $$M_{-1}(x)=\frac n{\frac1{x(0)}+\dots+\frac1{x(n-1)}}$$ of the values of functions $x\in\w^n$ which are not $0$-generating. The inequality $M_{-\infty}(x)\le M_{-1}(x)$, $x\in\w^n$, implies that $s_{-\infty}(n)\le s_{-1}(n)$ for all $n\in\IN$. So, it suffices to check that $s_{-1}(n)\le s_{-\infty}(n)$ for $n\le 4$. A vector $x\in\w^n$ will be called {\em monotone} if $x(i)\le x(j)$ for any $0\le i\le j<n$. It can be shown that a vector $x\in\w^n$ is $0$-generating if and only if some monotone vector $y\in x\circ\Sigma_n$ is $0$-generating.

\subsection{Case $n=2$} It can be shown that each monotone vector $x\in\w^2$ with $M_{-1}(x)>2$ is
greater or equal to the vector $(2,3)$. So, the inequality $s_{-1}(n)\le 2$ will follow as soon as we check that the vectors $(2,3)$ is $0$-generating. This is witnessed by the following annulating sequence:
\begin{table}[H]
\caption{A witness that the vector $(2,3)$ is $0$-generating}
\begin{tabular}
{|c|cc|c|cc|}\hline
$m$&$\hbar^{[m]}(0)$&$\hbar^{[m]}(1)$&$\phantom{\Big|}\sum_{i\in 2}\hbar^{[m]}(i)\phantom{\Big|}$&$\hbar^{\{m+1\}}(0)$&$\hbar^{\{m+1\}}(1)$\\ \hline
0&(1,0)&(0,1)&(1,1)& (0,1)& \\
1&(0,1)&(0,1)&(0,2)& &\bf{(0,0)} \\
\hline
\end{tabular}
\end{table}

\subsection{Case $n=3$} In this case consider the 3-element subset
$$A_3=\{(2,3,7),(2,4,5),(3,3,4)\}.$$

\begin{lemma}\label{l8.1} For each monotone vector $x\in\w^3$ with the harmonic mean $M_{-1}(x)>3$ there is a vector $y\in A_3$ such that $x\ge y$.
\end{lemma}

\begin{proof} It follows from $M_{-1}(x)>3$ that
$$\frac{1}{x(0)}+\frac1{x(1)}+\frac1{x(2)}<1.$$
This implies that $x(0)\ge 2$.

If $x(0)=2$, then the above inequality implies that
$\frac1{x(1)}+\frac1{x(2)}<1-\frac12=\frac12$ and hence $x(1)\ge 3$.
If $x(1)=3$, then $\frac1{x(2)}<\frac12-\frac13=\frac16$ and hence $x(2)\ge 7$. In this case we get  $x\ge (2,3,7)$.
If $x(1)=4$, then $\frac{1}{x(2)}<\frac12-\frac14=\frac14$ and $x(2)\ge 5$. In this case $x\ge (2,4,5)$. If $x(1)\ge 5$, then $x\ge (2,5,5)\ge (2,4,5)$.

If $x(0)=3$ and $x(1)=3$, then $\frac1{x(2)}<1-\frac23=\frac13$ and hence $x(1)\ge 4$. In this case $x\ge (3,3,4)$. If $x(0)=3$ and $x(1)\ge 4$, the $x\ge (3,4,4)\ge (3,3,4)$.
\end{proof}

By Lemma~\ref{l8.1} the upper bound $s_{-1}(3)\le 3$ will be proved as soon as we check that each vector $x\in A_3$ is $0$-generating. This is witnessed by the annulating sequences given in Tables~9--11.

\begin{table}[H]\label{tab:237}
\caption{A sequence witnessing that the vector $\hbar=(2,3,7)$ is $0$-generating}
\begin{tabular}{|c|ccc|c|ccc|}\hline
$m$&$\hbar^{[m]}(0)$&$\hbar^{[m]}(1)$&$\hbar^{[m]}(2)$&$\phantom{\Big|}\sum_{i\in 2}\hbar^{[m]}(i)\phantom{\Big|}$&$\hbar^{\{m+1\}}(0)$&$\hbar^{\{m+1\}}(1)$&$\hbar^{\{m+1\}}(2)$\\ \hline 0&(1,0,0)&(0,1,0)&(0,0,1)&(1,1,1)&(0,1,1)&  &  \\
1&(0,1,1)&(0,1,0)&(0,0,1)&(0,2,2)&   &(0,0,2)&   \\
2&(1,0,0)&(0,0,2)&(0,0,1)&(1,0,3)&(0,0,3)&   &   \\
3&(0,0,3)&(0,0,2)&(0,0,1)&(0,0,6)&   &   &\bf{(0,0,0)}\\
\hline
\end{tabular}
\end{table}

\begin{table}[H]
\caption{A sequence witnessing that the vector $\hbar=(2,4,5)$ is $0$-generating}
\begin{tabular}{|c|ccc|c|ccc|}\hline
$m$&$\hbar^{[m]}(0)$&$\hbar^{[m]}(1)$&$\hbar^{[m]}(2)$&$\phantom{\Big|}\sum_{i\in 2}\hbar^{[m]}(i)\phantom{\Big|}$&$\hbar^{\{m+1\}}(0)$&$\hbar^{\{m+1\}}(1)$&$\hbar^{\{m+1\}}(2)$\\ \hline 0&(1,0,0)&(0,1,0)&(0,0,1)&(1,1,1)&(0,1,1)& &    \\
1&(0,1,1)&(0,1,0)&(0,0,1)&(0,2,2)&  &(0,0,2)&   \\
2&(0,1,1)&(0,0,2)&(0,0,1)&(0,1,4)& & &(0,1,0)   \\
3&(1,0,0)&(0,1,0)&(0,1,0)&(1,2,0)&(0,2,0)& &    \\
4&(0,2,0)&(0,1,0)&(0,0,1)&(0,3,1)& &(0,0,1)  & \\
5&(1,0,0)&(0,0,1)&(0,0,1)&(1,0,2)&(0,0,2)& &   \\
6&(0,0,2)&(0,0,1)&(0,0,1)&(0,0,4)&   &   &\bf{(0,0,0)}\\
\hline
\end{tabular}
\end{table}

\begin{table}[H]\label{tab:334}
\caption{A sequence witnessing that the vector $\hbar=(3,3,4)$ is $0$-generating}
\begin{tabular}{|c|ccc|c|ccc|}\hline
$m$&$\hbar^{[m]}(0)$&$\hbar^{[m]}(1)$&$\hbar^{[m]}(2)$&$\phantom{\Big|}\sum_{i\in 2}\hbar^{[m]}(i)\phantom{\Big|}$&$\hbar^{\{m+1\}}(0)$&$\hbar^{\{m+1\}}(1)$&$\hbar^{\{m+1\}}(2)$\\ \hline 0&(1,0,0)&(0,1,0)&(0,0,1)&(1,1,1)&   &(1,0,1)&  \\
1&(1,0,0)&(1,0,1)&(0,0,1)&(2,0,2)&(0,0,2)&   &   \\
2&(0,0,2)&(0,1,0)&(0,0,1)&(0,1,3)&   &   &(0,1,0)    \\
3&(1,0,0)&(0,1,0)&(0,1,0)&(1,2,0)&   &(1,0,0)&   \\
4&(1,0,0)&(1,0,0)&(0,0,1)&(2,0,1)&(0,0,1)&   &    \\
5&(1,0,0)&(1,0,0)&(0,1,0)&(2,1,0)&(0,1,0)&   &      \\
6&(0,1,0)&(0,1,0)&(0,0,1)&(0,2,1)&   &(0,0,1)&      \\
7&(0,0,1)&(0,0,1)&(0,0,1)&(0,0,3)&   &   &\bf{(0,0,0)}   \\ \hline
\end{tabular}
\end{table}

\subsection{Case $n=4$}
Finally, we consider the case $n=4$. We should prove that $s_{-1}(4)\le 5$.
For this consider the following 11-element subset of $\w^4$
$$
\begin{aligned}
A_4=\{&(2,4,12,15), (2,5,9,13), (2,6,8,13), (2,7,7,11),(3,3,8,11),(3,4,5,12), (3,4,6,10),\\
 &(4,4,4,12), (4,4,5,9), (4,5,5,7), (4,5,6,6), (5,5,5,6)\}.
\end{aligned}
 $$
Each vector $x\in A_4$ is $0$-generating as witnessed by the annulating sequences presented in Tables~12--23 in Appendix. This fact combined with the following elementary lemma implies that  $s_{-1}(4)\le 5$.

\begin{lemma}\label{l8.2} For any monotone vector $x\in\w^4$ with $M_{-1}(x)>5$ there is a vector $y\in A_4$ such that $x\ge y$.
\end{lemma}

In the proof of this lemma we shall use another elementary lemma.

\begin{lemma}\label{l8.3} Let $x\le y$ be two positive integer numbers such that $\frac1x+\frac1y<a$ for some real number $a$. Then $(x,y)>(\frac1a,\frac2a)$.
\end{lemma}

\begin{proof}The inequality $x>a$ follows immediately from $\frac1x+\frac1y<a$. Since $x\le y$, we get $\frac2y\le\frac1x+\frac1y<a$ and hence $y>\frac2a$.
\end{proof}

\begin{proof}[Proof of Lemma~\ref{l8.2}] Given a monotone vector $x\in\w^4$ with $M_{-1}(x)>5$, we should find a vector $y\in A$ with $x\ge y$. Observe that the strict inequality $M_{-1}(x)>5$ is equivalent to
$$\frac1{x(0)}+\frac1{x(1)}+\frac1{x(2)}+\frac1{x(3)}<\frac45.$$
This implies $x(0)\ge 2$. Now we shall consider four cases:
\smallskip

1) $x(0)=2$. In this case we get
$$\frac1{x(1)}+\frac1{x(2)}+\frac1{x(3)}<\frac45-\frac12=\frac3{10},$$which implies
$x(1)\ge 4$. Now consider four subcases:
\smallskip

1a) If $x(1)=4$, then $\frac1{x(2)}+\frac1{x(3)}<\frac3{10}-\frac1{4}=\frac1{20}$ and $(x(2),x(3))\ge (21,41)$ according to Lemma~\ref{l8.3}. In this case
$x\ge (2,4,21,41)\ge (2,4,12,15)\in A_4$.
\smallskip

1b) If $x(1)=5$, then $\frac1{x(2)}+\frac1{x(3)}<\frac3{10}-\frac1{5}=\frac1{10}$ and 
$(x(2),x(3))\ge (11,21)$ according to Lemma~\ref{l8.3}. In this case
$x\ge (2,5,11,21)\ge (2,5,9,13)\in A_4$.
\smallskip

1c) If $x(1)=6$, then $\frac1{x(2)}+\frac1{x(3)}<\frac3{10}-\frac1{6}=\frac2{15}$ and  $(x(2),x(3))\ge (8,16)$ according to Lemma~\ref{l8.3}. In this case $x\ge (2,6,8,16)\ge (2,6,8,13)\in A_4$.
\smallskip

1d) If $x(1)\ge 7$, then $\frac1{x(2)}+\frac1{x(3)}<\frac3{10}-\frac1{7}=\frac{11}{70}$ and then $(x(2),x(3))\ge (7,13)$ according to Lemma~\ref{l8.3}. In this case $x\ge (2,7,7,13)\ge (2,7,7,11)\in A_4$.
\smallskip

2) $x(0)=3$. This case has two subcases.
\smallskip

2a) If $x(1)=3$, then $\frac1{x(2)}+\frac1{x(3)}<\frac45-\frac23=\frac2{15}$ and $(x(2),x(3))\ge (8, 16)$ according to Lemma~\ref{l8.3}. In this case $x\ge (3,3,8,16)\ge (3,3,8,11)\in A_4$.
\smallskip

2b) If $x(1)=4$ then $\frac1{x(2)}+\frac1{x(3)}<\frac45-\frac13-\frac14=\frac{13}{60}$ and hence $x(2)\ge 5$. If $x(2)=5$, then $\frac1{x(3)}<\frac{13}{60}-\frac15=\frac1{60}$ and $x\ge (3,4,5,61)\ge (3,4,5,12)\in A_4$. If $x(2)\ge 6$, then $\frac1{x(3)}<\frac{13}{60}-\frac16=\frac1{20}$ and $x\ge (3,4,6,21)\ge (3,4,6,10)\in A_4$.
\smallskip

3) $x(0)=4$. This case has three subcases.
\smallskip

3a) $x(1)=4$. If $x(2)=4$, then $\frac1{x(3)}<\frac45-\frac34=\frac1{20}$ and then $x\ge (4,4,4,21)\ge (4,4,4,12)\in A_4$. If $x(2)\ge 5$, then $\frac1{x(3)}<\frac45-\frac24-\frac15\le \frac1{10}$ and hence $x\ge (4,4,5,11)\ge (4,4,5,9)\in A_4$.
\smallskip

3b) $x(1)=5$. If $x(2)=5$, then $\frac1{x(3)}<\frac45-\frac14-\frac25=\frac3{20}$ and $x\ge(4,5,5,7)\in A_4$. If $x(2)\ge 6$, then $x\ge(4,5,6,6)\in A_4$.
\smallskip

3c) $x(1)\ge 6$ In this case $x\ge (4,6,6,6)\ge (4,5,6,6)\in A_4$.
\smallskip

4) $x(0)=5$. In this case the inequality $M_{-1}(x)>5$ implies $x\ge (5,5,5,6)\in A_4$.
\end{proof}

\section{Acknowledgements}

The authors express their sincere thanks to Igor Protasov and Ostap Chervak for valuable discussions on the topic of this paper.

%\newpage

\appendix
\section{Computer Assisted Proofs of $0$-generacy of some sequences}

\begin{table}[H]\label{tab:24}
\caption{A sequence witnessing that the function $\hbar=(2,4,12,15)$ is $0$-generating}
\begin{tabular}{|c|cccc|c|cccc|}\hline
$m$&$\hbar^{[m]}(0)$&$\hbar^{[m]}(1)$&$\hbar^{[m]}(2)$&$\hbar^{[m]}(3)$&$\phantom{\Big|}\sum_{i\in 3}\hbar^{[m]}(i)\phantom{\Big|}$&$\hbar^{\{m+1\}}(0)$&$\hbar^{\{m+1\}}(1)$&$\hbar^{\{m+1\}}(2)$&
$\hbar^{\{m+1\}}(3)$ \\ \hline
0 & (1,0,0,0) & (0,1,0,0) & (0,0,1,0) & (0,0,0,1) & (1,1,1,1) & (0,1,1,1) & & & \\
1 & (0,1,1,1) & (0,1,0,0) & (0,0,1,0) & (0,0,0,1) & (0,2,2,2) & & (0,0,2,2) & & \\
2 & (1,0,0,0) & (0,0,2,2) & (0,0,1,0) & (0,0,0,1) & (1,0,3,3) & (0,0,3,3) & & & \\
3 & (0,1,1,1) & (0,0,2,2) & (0,0,1,0) & (0,0,0,1) & (0,1,4,4) & & & (0,1,0,4) & \\
4 & (0,1,1,1) & (0,1,0,0) & (0,1,0,4) & (0,0,0,1) & (0,3,1,6) & & (0,0,1,6) & & \\
5 & (0,1,1,1) & (0,0,1,6) & (0,0,1,0) & (0,0,0,1) & (0,1,3,8) & & & & (0,1,3,0) \\
6 & (0,1,1,1) & (0,1,0,0) & (0,0,1,0) & (0,1,3,0) & (0,3,5,1) & & (0,0,5,1) & & \\
\hline
\end{tabular}
\end{table}

\begin{table}[H]
\caption{A sequence witnessing that the function $\hbar=(2,5,9,13)$ is $0$-generating}
\small
\begin{tabular}{|c|cccc|c|cccc|}\hline
$m$&$\hbar^{[m]}(0)$&$\hbar^{[m]}(1)$&$\hbar^{[m]}(2)$&$\hbar^{[m]}(3)$&$\phantom{\Big|}\sum_{i\in 3}\hbar^{[m]}(i)\phantom{\Big|}$&$\hbar^{\{m+1\}}(0)$&$\hbar^{\{m+1\}}(1)$&$\hbar^{\{m+1\}}(2)$&
$\hbar^{\{m+1\}}(3)$ \\ \hline
0&(1,0,0,0)&(0,1,0,0)&(0,0,1,0)&(0,0,0,1)&(1,1,1,1)&(0,1,1,1) &  &  &  \\
1&(0,1,1,1)&(0,1,0,0)&(0,0,1,0)&(0,0,0,1)&(0,2,2,2)& &  &  & (0,2,2,0) \\
2&(0,1,1,1)&(0,1,0,0)&(0,0,1,0)&(0,2,2,0)&(0,4,4,1)& & (0,0,4,1) &  &  \\
3&(0,1,1,1)&(0,0,4,1)&(0,0,1,0)&(0,0,0,1)&(0,1,6,3)& &  & (0,1,0,3) &  \\
4&(1,0,0,0)&(0,1,0,0)&(0,1,0,3)&(0,0,0,1)&(1,2,0,4)&(0,2,0,4) &  &  &  \\
5&(0,2,0,4)&(0,1,0,0)&(0,1,0,3)&(0,0,0,1)&(0,4,0,8)& & (0,0,0,8) &  &  \\
6&(0,1,1,1)&(0,0,0,8)&(0,0,1,0)&(0,0,0,1)&(0,1,2,10)& &  &  & (0,1,2,0) \\
\hline
\end{tabular}
\end{table}

\begin{table}[H]
\caption{A sequence witnessing that the function $\hbar=(2,6,8,13)$ is $0$-generating}
{\small
\begin{tabular}{|c|cccc|c|cccc|}\hline
$m$&$\hbar^{[m]}(0)$&$\hbar^{[m]}(1)$&$\hbar^{[m]}(2)$&$\hbar^{[m]}(3)$&$\phantom{\Big|}\sum_{i\in 3}\hbar^{[m]}(i)\phantom{\Big|}$&$\hbar^{\{m+1\}}(0)$&$\hbar^{\{m+1\}}(1)$&$\hbar^{\{m+1\}}(2)$&
$\hbar^{\{m+1\}}(3)$ \\ \hline
0&(1,0,0,0)&(0,1,0,0)&(0,0,1,0)&(0,0,0,1)&(1,1,1,1)&(0,1,1,1) &  &  &  \\
1&(0,1,1,1)&(0,1,0,0)&(0,0,1,0)&(0,0,0,1)&(0,2,2,2)& & (0,0,2,2) & (0,2,0,2) &  \\
2&(0,1,1,1)&(0,0,2,2)&(0,0,1,0)&(0,0,0,1)&(0,1,4,4)& &  & (0,1,0,4) & (0,1,4,0) \\
3&(0,1,1,1)&(0,1,0,0)&(0,0,1,0)&(0,1,4,0)&(0,3,6,1)& &  & (0,3,0,1) &  \\
4&(0,1,1,1)&(0,1,0,0)&(0,3,0,1)&(0,0,0,1)&(0,5,1,3)& & (0,0,1,3) &  &  \\
5&(0,1,1,1)&(0,1,0,0)&(0,1,0,4)&(0,0,0,1)&(0,3,1,6)& &  &  & (0,3,1,0) \\
6&(0,1,1,1)&(0,1,0,0)&(0,0,1,0)&(0,3,1,0)&(0,5,3,1)& & (0,0,3,1) &  &  \\
\hline
\end{tabular}
}
\end{table}

\begin{table}[H]
\caption{A sequence witnessing that the function $\hbar=(2,7,7,11)$ is $0$-generating}
{\small
\begin{tabular}{|c|cccc|c|cccc|}\hline
$m$&$\hbar^{[m]}(0)$&$\hbar^{[m]}(1)$&$\hbar^{[m]}(2)$&$\hbar^{[m]}(3)$&$\phantom{\Big|}\sum_{i\in 3}\hbar^{[m]}(i)\phantom{\Big|}$&$\hbar^{\{m+1\}}(0)$&$\hbar^{\{m+1\}}(1)$&$\hbar^{\{m+1\}}(2)$&
$\hbar^{\{m+1\}}(3)$ \\ \hline
0&(1,0,0,0)&(0,1,0,0)&(0,0,1,0)&(0,0,0,1)&(1,1,1,1)&(0,1,1,1) &  &  &  \\
1&(0,1,1,1)&(0,1,0,0)&(0,0,1,0)&(0,0,0,1)&(0,2,2,2)& &  & (0,2,0,2) &  \\
2&(0,1,1,1)&(0,1,0,0)&(0,2,0,2)&(0,0,0,1)&(0,4,1,4)& &  &  & (0,4,1,0) \\
3&(0,1,1,1)&(0,1,0,0)&(0,0,1,0)&(0,4,1,0)&(0,6,3,1)& & (0,0,3,1) &  &  \\
4&(0,1,1,1)&(0,0,3,1)&(0,0,1,0)&(0,0,0,1)&(0,1,5,3)& &  & (0,1,0,3) &  \\
5&(1,0,0,0)&(0,1,0,0)&(0,1,0,3)&(0,0,0,1)&(1,2,0,4)&(0,2,0,4) &  &  &  \\
6&(0,2,0,4)&(0,1,0,0)&(0,1,0,3)&(0,0,0,1)&(0,4,0,8)& & (0,0,0,8) &  & (0,4,0,0) \\
\hline
\end{tabular}
}
\end{table}

\begin{table}[H]
\caption{A sequence witnessing that the function $\hbar=(3,3,8,11)$ is $0$-generating}
{\small
\begin{tabular}{|c|cccc|c|cccc|}\hline
$m$&$\hbar^{[m]}(0)$&$\hbar^{[m]}(1)$&$\hbar^{[m]}(2)$&$\hbar^{[m]}(3)$&$\phantom{\Big|}\sum_{i\in 3}\hbar^{[m]}(i)\phantom{\Big|}$&$\hbar^{\{m+1\}}(0)$&$\hbar^{\{m+1\}}(1)$&$\hbar^{\{m+1\}}(2)$&
$\hbar^{\{m+1\}}(3)$ \\ \hline
0&(1,0,0,0)&(0,1,0,0)&(0,0,1,0)&(0,0,0,1)&(1,1,1,1)&(0,1,1,1) & (1,0,1,1) &  &  \\
1&(1,0,0,0)&(1,0,1,1)&(0,0,1,0)&(0,0,0,1)&(2,0,2,2)&(0,0,2,2) &  &  &  \\
2&(0,1,1,1)&(0,1,0,0)&(0,0,1,0)&(0,0,0,1)&(0,2,2,2)& & (0,0,2,2) &  &  \\
3&(0,0,2,2)&(0,0,2,2)&(0,0,1,0)&(0,0,0,1)&(0,0,5,5)& &  & (0,0,0,5) &  \\
4&(1,0,0,0)&(0,0,2,2)&(0,0,0,5)&(0,0,0,1)&(1,0,2,8)& &  &  & (1,0,2,0) \\
5&(1,0,0,0)&(0,1,0,0)&(0,0,1,0)&(1,0,2,0)&(2,1,3,0)&(0,1,3,0) &  &  &  \\
6&(0,1,3,0)&(0,1,0,0)&(0,0,1,0)&(0,0,0,1)&(0,2,4,1)& & (0,0,4,1) &  &  \\
\hline
\end{tabular}
}
\end{table}

\begin{table}[H]
\caption{A sequence witnessing that the function $\hbar=(3,4,5,12)$ is $0$-generating}
{\small
\begin{tabular}{|c|cccc|c|cccc|}\hline
$m$&$\hbar^{[m]}(0)$&$\hbar^{[m]}(1)$&$\hbar^{[m]}(2)$&$\hbar^{[m]}(3)$&$\phantom{\Big|}\sum_{i\in 3}\hbar^{[m]}(i)\phantom{\Big|}$&$\hbar^{\{m+1\}}(0)$&$\hbar^{\{m+1\}}(1)$&$\hbar^{\{m+1\}}(2)$&
$\hbar^{\{m+1\}}(3)$ \\ \hline
0&(1,0,0,0)&(0,1,0,0)&(0,0,1,0)&(0,0,0,1)&(1,1,1,1)&(0,1,1,1) &  &  &  \\
1&(0,1,1,1)&(0,1,0,0)&(0,0,1,0)&(0,0,0,1)&(0,2,2,2)& & (0,0,2,2) &  &  \\
2&(0,1,1,1)&(0,0,2,2)&(0,0,1,0)&(0,0,0,1)&(0,1,4,4)& &  & (0,1,0,4) &  \\
3&(0,1,1,1)&(0,1,0,0)&(0,1,0,4)&(0,0,0,1)&(0,3,1,6)& & (0,0,1,6) &  &  \\
4&(1,0,0,0)&(0,0,1,6)&(0,0,1,0)&(0,0,0,1)&(1,0,2,7)& &  &  & (1,0,2,0) \\
5&(1,0,0,0)&(0,1,0,0)&(0,0,1,0)&(1,0,2,0)&(2,1,3,0)&(0,1,3,0) &  &  &  \\
6&(0,1,3,0)&(0,1,0,0)&(0,0,1,0)&(0,0,0,1)&(0,2,4,1)& &  & (0,2,0,1) &  \\
\hline
\end{tabular}
}
\end{table}

\begin{table}[H]
\caption{A sequence witnessing that the function $\hbar=(3,4,6,10)$ is $0$-generating}
\small
\begin{tabular}{|c|cccc|c|cccc|}\hline
$m$&$\hbar^{[m]}(0)$&$\hbar^{[m]}(1)$&$\hbar^{[m]}(2)$&$\hbar^{[m]}(3)$&$\phantom{\Big|}\sum_{i\in 3}\hbar^{[m]}(i)\phantom{\Big|}$&$\hbar^{\{m+1\}}(0)$&$\hbar^{\{m+1\}}(1)$&$\hbar^{\{m+1\}}(2)$&
$\hbar^{\{m+1\}}(3)$ \\ \hline
0&(1,0,0,0)&(0,1,0,0)&(0,0,1,0)&(0,0,0,1)&(1,1,1,1)&(0,1,1,1) &  &  &  \\
1&(0,1,1,1)&(0,1,0,0)&(0,0,1,0)&(0,0,0,1)&(0,2,2,2)& & (0,0,2,2) &  &  \\
2&(0,1,1,1)&(0,0,2,2)&(0,0,1,0)&(0,0,0,1)&(0,1,4,4)& &  & (0,1,0,4) &  \\
3&(0,1,1,1)&(0,1,0,0)&(0,1,0,4)&(0,0,0,1)&(0,3,1,6)& & (0,0,1,6) &  &  \\
4&(1,0,0,0)&(0,0,1,6)&(0,0,1,0)&(0,0,0,1)&(1,0,2,7)& &  &  & (1,0,2,0) \\
5&(1,0,0,0)&(0,1,0,0)&(0,0,1,0)&(1,0,2,0)&(2,1,3,0)&(0,1,3,0) &  &  &  \\
6&(0,1,3,0)&(0,1,0,0)&(0,0,1,0)&(0,0,0,1)&(0,2,4,1)& & (0,0,4,1) &  &  \\
\hline
\end{tabular}
\end{table}

\begin{table}[H]
\caption{A sequence witnessing that the function $\hbar=(4,4,4,12)$ is $0$-generating}
{\small
\begin{tabular}{|c|cccc|c|cccc|}\hline
$m$&$\hbar^{[m]}(0)$&$\hbar^{[m]}(1)$&$\hbar^{[m]}(2)$&$\hbar^{[m]}(3)$&$\phantom{\Big|}\sum_{i\in 3}\hbar^{[m]}(i)\phantom{\Big|}$&$\hbar^{\{m+1\}}(0)$&$\hbar^{\{m+1\}}(1)$&$\hbar^{\{m+1\}}(2)$&
$\hbar^{\{m+1\}}(3)$ \\ \hline
0&(1,0,0,0)&(0,1,0,0)&(0,0,1,0)&(0,0,0,1)&(1,1,1,1)&(0,1,1,1) &  & (1,1,0,1) &  \\
1&(1,0,0,0)&(0,1,0,0)&(1,1,0,1)&(0,0,0,1)&(2,2,0,2)& & (2,0,0,2) &  &  \\
2&(1,0,0,0)&(2,0,0,2)&(0,0,1,0)&(0,0,0,1)&(3,0,1,3)&(0,0,1,3) &  &  &  \\
3&(0,1,1,1)&(0,1,0,0)&(0,0,1,0)&(0,0,0,1)&(0,2,2,2)& &  & (0,2,0,2) &  \\
4&(1,0,0,0)&(0,1,0,0)&(0,2,0,2)&(0,0,0,1)&(1,3,0,3)& & (1,0,0,3) &  &  \\
5&(1,0,0,0)&(1,0,0,3)&(0,0,1,0)&(0,0,0,1)&(2,0,1,4)& &  &  & (2,0,1,0) \\
6&(1,0,0,0)&(0,1,0,0)&(0,0,1,0)&(2,0,1,0)&(3,1,2,0)&(0,1,2,0) &  &  &  \\
\hline
\end{tabular}
}
\end{table}

\begin{table}[H]
\caption{A sequence witnessing that the function $\hbar=(4,4,5,9)$ is $0$-generating}
{\small
\begin{tabular}{|c|cccc|c|cccc|}\hline
$m$&$\hbar^{[m]}(0)$&$\hbar^{[m]}(1)$&$\hbar^{[m]}(2)$&$\hbar^{[m]}(3)$&$\phantom{\Big|}\sum_{i\in 3}\hbar^{[m]}(i)\phantom{\Big|}$&$\hbar^{\{m+1\}}(0)$&$\hbar^{\{m+1\}}(1)$&$\hbar^{\{m+1\}}(2)$&
$\hbar^{\{m+1\}}(3)$ \\ \hline
0&(1,0,0,0)&(0,1,0,0)&(0,0,1,0)&(0,0,0,1)&(1,1,1,1)&(0,1,1,1) &  &  &  \\
1&(0,1,1,1)&(0,1,0,0)&(0,0,1,0)&(0,0,0,1)&(0,2,2,2)& & (0,0,2,2) & (0,2,0,2) & (0,2,2,0) \\
2&(1,0,0,0)&(0,1,0,0)&(0,0,1,0)&(0,2,2,0)&(1,3,3,0)& & (1,0,3,0) &  &  \\
3&(1,0,0,0)&(0,1,0,0)&(0,2,0,2)&(0,0,0,1)&(1,3,0,3)& & (1,0,0,3) &  &  \\
4&(1,0,0,0)&(1,0,3,0)&(0,0,1,0)&(0,0,0,1)&(2,0,4,1)& &  & (2,0,0,1) &  \\
5&(1,0,0,0)&(0,1,0,0)&(2,0,0,1)&(0,0,0,1)&(3,1,0,2)&(0,1,0,2) &  &  &  \\
6&(1,0,0,0)&(0,0,2,2)&(0,0,1,0)&(0,0,0,1)&(1,0,3,3)& &  & (1,0,0,3) & (1,0,3,0) \\
\hline
\end{tabular}
}
\end{table}

\begin{table}[H]
\caption{A sequence witnessing that the function $\hbar=(4,5,5,7)$ is $0$-generating}
{\small
\begin{tabular}{|c|cccc|c|cccc|}\hline
$m$&$\hbar^{[m]}(0)$&$\hbar^{[m]}(1)$&$\hbar^{[m]}(2)$&$\hbar^{[m]}(3)$&$\phantom{\Big|}\sum_{i\in 3}\hbar^{[m]}(i)\phantom{\Big|}$&$\hbar^{\{m+1\}}(0)$&$\hbar^{\{m+1\}}(1)$&$\hbar^{\{m+1\}}(2)$&
$\hbar^{\{m+1\}}(3)$ \\ \hline
0&(1,0,0,0)&(0,1,0,0)&(0,0,1,0)&(0,0,0,1)&(1,1,1,1)&(0,1,1,1) &  &  &  \\
1&(0,1,1,1)&(0,1,0,0)&(0,0,1,0)&(0,0,0,1)&(0,2,2,2)& & (0,0,2,2) & (0,2,0,2) & (0,2,2,0) \\
2&(1,0,0,0)&(0,1,0,0)&(0,0,1,0)&(0,2,2,0)&(1,3,3,0)& & (1,0,3,0) &  &  \\
3&(1,0,0,0)&(0,1,0,0)&(0,2,0,2)&(0,0,0,1)&(1,3,0,3)& & (1,0,0,3) &  &  \\
4&(1,0,0,0)&(1,0,0,3)&(0,0,1,0)&(0,0,0,1)&(2,0,1,4)& &  &  & (2,0,1,0) \\
5&(1,0,0,0)&(0,1,0,0)&(0,0,1,0)&(2,0,1,0)&(3,1,2,0)&(0,1,2,0) &  &  &  \\
6&(1,0,0,0)&(1,0,3,0)&(0,0,1,0)&(0,0,0,1)&(2,0,4,1)& &  & (2,0,0,1) &  \\
\hline
\end{tabular}
}
\end{table}

\begin{table}[H]\label{tab:5556}
\caption{A sequence witnessing that the function $\hbar=(5,5,5,6)$ is $0$-generating}
{\small
\begin{tabular}{|c|cccc|c|cccc|}\hline
$m$&$\hbar^{[m]}(0)$&$\hbar^{[m]}(1)$&$\hbar^{[m]}(2)$&$\hbar^{[m]}(3)$&$\phantom{\Big|}\sum_{i\in 3}\hbar^{[m]}(i)\phantom{\Big|}$&$\hbar^{\{m+1\}}(0)$&$\hbar^{\{m+1\}}(1)$&$\hbar^{\{m+1\}}(2)$&
$\hbar^{\{m+1\}}(3)$ \\ \hline
0&(1,0,0,0)&(0,1,0,0)&(0,0,1,0)&(0,0,0,1)&(1,1,1,1)& &  & (1,1,0,1) & (1,1,1,0) \\
1&(1,0,0,0)&(0,1,0,0)&(0,0,1,0)&(1,1,1,0)&(2,2,2,0)& & (2,0,2,0) &  &  \\
2&(1,0,0,0)&(0,1,0,0)&(1,1,0,1)&(0,0,0,1)&(2,2,0,2)& & (2,0,0,2) &  &  \\
3&(1,0,0,0)&(2,0,0,2)&(0,0,1,0)&(0,0,0,1)&(3,0,1,3)& &  &  & (3,0,1,0) \\
4&(1,0,0,0)&(0,1,0,0)&(0,0,1,0)&(3,0,1,0)&(4,1,2,0)&(0,1,2,0) &  &  &  \\
5&(1,0,0,0)&(2,0,2,0)&(0,0,1,0)&(0,0,0,1)&(3,0,3,1)& &  & (3,0,0,1) &  \\
6&(1,0,0,0)&(0,1,0,0)&(3,0,0,1)&(0,0,0,1)&(4,1,0,2)&(0,1,0,2) &  &  &  \\
\hline
\end{tabular}
}
\end{table}

\begin{table}[H]
\caption{A sequence witnessing that the function $\hbar=(4,5,6,6)$ is $0$-generating}
{\small
\begin{tabular}{|c|cccc|c|cccc|}\hline
$m$&$\hbar^{[m]}(0)$&$\hbar^{[m]}(1)$&$\hbar^{[m]}(2)$&$\hbar^{[m]}(3)$&$\phantom{\Big|}\sum_{i\in 3}\hbar^{[m]}(i)\phantom{\Big|}$&$\hbar^{\{m+1\}}(0)$&$\hbar^{\{m+1\}}(1)$&$\hbar^{\{m+1\}}(2)$&
$\hbar^{\{m+1\}}(3)$ \\ \hline
0&(1,0,0,0)&(0,1,0,0)&(0,0,1,0)&(0,0,0,1)&(1,1,1,1)&(0,1,1,1) &  &  &  \\
1&(0,1,1,1)&(0,1,0,0)&(0,0,1,0)&(0,0,0,1)&(0,2,2,2)& & (0,0,2,2) & (0,2,0,2) & (0,2,2,0) \\
2&(1,0,0,0)&(0,1,0,0)&(0,0,1,0)&(0,2,2,0)&(1,3,3,0)& & (1,0,3,0) &  &  \\
3&(1,0,0,0)&(0,1,0,0)&(0,2,0,2)&(0,0,0,1)&(1,3,0,3)& & (1,0,0,3) &  &  \\
4&(1,0,0,0)&(1,0,0,3)&(0,0,1,0)&(0,0,0,1)&(2,0,1,4)& &  &  & (2,0,1,0) \\
5&(1,0,0,0)&(0,1,0,0)&(0,0,1,0)&(2,0,1,0)&(3,1,2,0)&(0,1,2,0) &  &  &  \\
6&(1,0,0,0)&(1,0,3,0)&(0,0,1,0)&(0,0,0,1)&(2,0,4,1)& &  & (2,0,0,1) &  \\
\hline
\end{tabular}
}
\end{table}

\end{document}